\documentclass[reqno]{amsart}
\usepackage{PianoPreamble}
\usepackage[backend=bibtex,bibstyle=alphabetic,citestyle=alphabetic]{biblatex}
\DeclareNoinit{
	\noinit{\regexp{\v{S}\c{C}}}
}

\usepackage{upref}
\crefdefaultlabelformat{#2\textup{#1}#3}

\title[Piano Algebras]{Extended Admissible Dissections of Marked Surfaces and Piano Algebras}
\author{Marina Godinho and Dave Murphy}
\date{}

\begin{document}

\begin{abstract}
	We introduce the notion of extended admissible dissections of a marked surface, building upon the notion of an admissible dissection of a marked surface by Amiot-Plamondon-Schroll.
	For each extended admissible dissection we construct a differential graded algebra, called a piano algebra.
	We show that for a marked disc without punctures, the homology of a piano algebra is isomorphic to the graded endomorphism ring of a classical generator of the Paquette-Y\i ld\i r\i m completion of the discrete cluster category of Dynkin type $A_{\infty}$, labelled $\overline{\mathcal{C}}_n$.
	
	Furthermore, we show that if two piano algebras $\Lambda$ and $\Gamma$ both correspond to extended admissible dissections of homeomorphic unpunctured marked discs, then $\Lambda$ and $\Gamma$ are derived equivalent as differential graded algebras.
\end{abstract}

\maketitle

\tableofcontents

\section*{Introduction}

The use of combinatoric models to understand certain categories has become a popular trend in recent years.
The concept is based in the notion of being able to visualise certain data from a given category in the hopes to understand the representation theory of the category via more combinatoric means.
Usually this involves viewing indecomposable objects as combinatoric objects, and showing that certain homological properties are in correspondence with some combinatoric properties.

Two particular classes of algebras that this process has proven particularly powerful in are cluster algebras and gentle algebras.
Combinatoric models for categories associated to these algebras (e.g.\ cluster categories, module categories, bounded derived categories) often take the form of \textit{marked surfaces}, that is, the data of a compact, orientable surface (often with boundary) and a set of \textit{marked points} on the surface.
In these models, indecomposable objects correspond to isotopy classes of curves on the surface with endpoints in the set of marked points, and the dimension of $\mathrm{Hom}$-spaces are in some correspondence with number of times corresponding arcs cross, for some suitable definition of crossing.

Cluster algebras were introduced by Fomin and Zelevinsky in a series of papers \cite{FominZelevinsky1,FominZelevinsky2,ClusterAlg3,FominZelevinsky4} around the start of the millennium.
They are typified by their \textit{clusters}, certain sets of generators, and their \textit{cluster mutation}, a process of uniquely replacing a given generator in a cluster  to obtain a new cluster, according to some prescribed rule.
It was noticed in the first paper by Fomin and Zelevinksy \cite{FominZelevinsky1} that clusters in the cluster algebra of Dynkin type $A_n$ were in correspondence with the triangulations of an $(n+3)$-gon, and the cluster mutation corresponded to flipping a diagonal in the triangulation.

Cluster categories are a family of triangulated categories, and their study began with \cite{BMRRT,CCS}, where both sets of authors independently sought to categorify the then recently developed theory of cluster algebras.
In \cite{BMRRT}, the authors aimed to be able to study the tilting theory of an hereditary algebra by categorical means.
To do this, for an hereditary algebra $H$, they consider the (triangulated) orbit category of the bounded derived category $D^b(H)/( \tau^{-1} [1])$, where $\tau$ is the Auslander-Reiten translation, and $[1]$ denotes the suspension functor.
 
It was shown in \cite{CCS} that the cluster category of $kA_n$ may be modelled by the $(n+3)$-gon, with indecomposable objects corresponding to diagonals, non-trivial $\mathrm{Ext}^1$-spaces between indecomposable objects corresponding to the respective diagonals crossing, and the triangulations corresponded to the so-called cluster-tilting subcategories.

In \cite{Holm2009}, the authors studied $\mathrm{D}^{\mathrm{f}}(k[x])$, the finite derived category of the polynomial ring $(k[x],d=0)$ seen as a differential graded algebra (=\textit{dg algebra}) with $x$ placed in degree $-1$.
The authors showed that the category was modelled by the $\infty$-gon, more commonly viewed as a disc with $\mathbb{Z}$ many marked points on the boundary, and a single \textit{accumulation point}, which is not considered to be a marked point here.

A completed version of these categories were studied by Fisher \cite{Fisher2014}, where they look at the closure under certain homotopy colimits.
The author shows that the resulting category is triangulated and may be modelled by the $\infty$-gon as well, however we now consider the accumulation point to be a marked point, and indecomposable homotopy colimits correspond to those arcs with the accumulation point as an endpoint.
It was proven in \cite{ACFGS} that this category is actually triangle equivalent to the perfect derived category of the dg algebra, $(k[x],d=0)$, previously considered by Holm and J\o rgensen.

Igusa and Todorov provided an algorithm in \cite{Igusa2013} to begin with a cyclic poset, and to obtain a cluster category.
This provided a combinatoric generalisation to the category studied by Holm and J\o rgensen in \cite{Holm2009}, by constructing triangulated categories that could be modelled by the $\infty$-gon with an arbitrary number of finite accumulation points.
These are algebraic, Krull-Schmidt, 2-Calabi-Yau, Hom-finite, k-linear, triangulated categories, with classical generators \cite{Generators}, and as such are equivalent as triangulated categories to the perfect derived category of some differential graded algebra by a result of Keller \cite{KellerDeriving}.
These categories exhibit interesting phenomena that is not present in the finite rank cases, and this has lead to a lot of recent interest in them \cite{CummingsGratz,FranchiniNegative,Gratz2017,GZ2021,Holm2012,MurphyGrothendieck}.

Paquette and Y\i ld\i r\i m \cite{Paquette2020} constructed a family of categories, as a Verdier localisation of the Igusa-Todorov categories that also complete the Igusa-Todorov categories in the sense that they can be modelled by now considering the accumulation points as marked points too.
We denote these categories by $\ocC$ for $n \geq 1$, where $n$ is the number of accumulation points in the combinatoric model, and call them \textit{Paquette-Y\i ld\i r\i m categories}.
These categories have exceptionally powerful combinatoric properties, and can be used to understand develop combinatoric techniques to study the representation theory of triangulated categories with cluster-tilting subcategories containing infinite indecomposable objects up to isomorphism.
Just as in the non-completed case, there has been quite a bit recent work on the category $\ocC$ \cite{CanakciKalckPressland,Franchini,Generators,MurphyGrothendieck}.

However, a natural question now arises, can we describe the dg algebra $\Lambda$ such that $\ocC$ is triangle equivalent to the perfect derived category of $\Lambda$?

To begin our work on a finding a dg algebra, we turn to the combinatoric models introduced in \cite{AmiotPlamondonSchroll}, which they use to model the bounded derived category of a gentle algebra.
In this model $(\Sigma,M,P)$, we consider an orientable surface $\Sigma$, possibly with boundary,and two sets of coloured marked points, one coloured red $M_{\rcirc} \cup P_{\rcirc}$ and the other green $M_{\gbullet} \cup P_{\gbullet}$, such that $M_{\rcirc} \cup M_{\gbullet}$ is contained on the boundary of $\Sigma$.
An \textit{admissible dissection} of the surface is a maximal set of non-crossing arcs between points in $M_{\rcirc} \cup P_{\rcirc}$, such that they cut out a topological disc with exactly one point in $M_{\gbullet} \cup P_{\gbullet}$, either in the interior or on the boundary.
It was shown in \cite{OpperPlamondonSchroll}, for an equivalent combinatoric model, that the class of admissible dissections up to homeomorphism are in bijection to the equivalences classes of gentle algebras.
It should also be noted that the combinatoric model for the module category of a gentle algebra was provided by Baur and Coelho Sim\~{o}es in \cite{BaurRaquel}.

We introduce \textit{extended admissible dissections} of a marked surface, in which we take the data of an admissible dissection, and add in a single so-called \textit{binding arc} to each green marked point, connecting a green marked point to a red one.
We again ask that this collection of arcs is pairwise non-crossing.
To each extended admissible dissection $\Delta$, we construct a dg algebra, which we call the \textit{piano algebra} of $\Delta$, which we label $\Lambda^{\Delta}$.

This leads us to our primary result on extended admissible dissections, here $D^2$ is a disc, and $M_n$ is a set of $2n$ marked points on the boundary of $D^2$.

\begin{theorem*}[\Cref{Prop: Bijection of ext admissible dissections} and \Cref{Thm:path algebras}]\label{Intro Thm1}
	There is a bijection between the homeomorphism classes of extended admissible dissections of $(D^2,M_n,\emptyset)$, and the limit generators of $\ocC$.
	Moreover, let $\Delta$ be an extended admissible dissection of $(D^2,M_n,\emptyset)$.
	Then there exists a classical generator $G \in \ocC$ such that $\mathrm{H}^{\ast}(\Lambda^{\Delta}) \cong \Endo[\ocC]{\ast}{G}$ as graded $k$-algebras.
\end{theorem*}

Limit generators are a special class of classical generators of $\ocC$, signified by some particularly nice homological properties.
Their connections to extended admissible dissections of $(D^2,M_n,\emptyset)$ makes them particularly powerful classical generators for our particular interests.
In fact, this connection provides us with the following result.

\begin{theorem*}[\Cref{Thm:path algebras} and \Cref{Lem: Lambda 0 is gentle}]
	Let $G$ be a limit generator of $\ocC$.
	Then there exists a graded quiver with relations $(Q^G,I^G)$ such that there is an isomorphism of graded $k$-algebras,
	\[
	\Endo[\ocC]{\ast}{G} \xrightarrow{\sim} kQ^G/I^G.
	\]
	Moreover, $\Endo[\ocC]{}{G}$ is isomorphic to a locally gentle algebra.
\end{theorem*}

A key result of \cite{AmiotPlamondonSchroll} is that two gentle algebras are derived equivalent if and only if there exists an orientation-preserving homeomorphism between their associated marked surfaces, along with a condition concerning the line fields on the marked surfaces.
It is natural for us then to ask whether a similar result may hold for their respective piano algebras.
That is, are two piano algebras derived equivalent if and only if there exists an orientation-preserving homeomorphism between their respective associated marked surfaces?

We do not give a conclusive answer to this question in this paper, however we do begin to study this question by looking at the piano algebras associated to the marked disc $(D^2,M_n,\emptyset)$.
However, to do this we need to make use of the connection between these piano algebras and the limit generators of $\ocC$.

It was shown in \cite{LinearGens} that there is a dg algebra $\Lambda_n$ such that $\ocC$ and $\mathrm{perf}\Lambda_n$ are equivalent as triangulated categories.
Here the dg algebra $\Lambda_n$ has trivial differential, and is quasi-isomorphic to the graded endomorphism ring of a so-called \textit{fan generator} of $\ocC$.
However, we show that a fan generator of $\ocC$ is in fact a special case of a limit generator, and so $\Lambda_n$ is a piano algebra.

Given the equivalence of triangulated categories in \cite{LinearGens}, we can view a piano algebra $\Lambda^{\Delta}$, that is associated to an unpunctured disc, as a $\Lambda^{\Delta}-\Lambda_n$ bimodule.
This view point allows us to use a result due to Keller in \cite{KellerDeriving} to prove our final result.

\begin{theorem*}[\Cref{Thm: Derived Equiv}]
	There is a equivalence of triangulated categories
	\[
	D(\Lambda^{\Delta}) \xrightarrow{\sim} D(\Lambda_n),
	\]
	for any extended admissible dissection $\Delta$ of the marked surface $(D^2,M_n,\emptyset)$.
\end{theorem*}

In \Cref{Sec: Those cats} we recall the Paquette-Y\i ld\i r\i m categories, denoted $\ocC$ for all $n \geq 1$, and some necessary results about them.
In particular we recall the classification of their classical generators and introduce the notion of a limit generator.

\Cref{Sec: Dissections} is where we recall the construction of admissible dissections of marked surfaces by Amiot, Plamondon, and Schroll \cite{AmiotPlamondonSchroll}.
We also recall how to construct gentle algebras from these admissible dissections.
We then begin our work here by introducing extended admissible dissections, and proving some elementary results involving extended admissible dissections.
Finally, we introduce the piano algebra associated to an extended admissible dissection.

We show that there exists a strong connection between piano algebras associated to the marked surface $(D^2,M_n,\emptyset)$, and the limit generators of $\ocC$ in \Cref{Sec:Gens}.
In fact we show that there exists an isomorphism of graded algebras between the graded endomorphism ring of a limit generator and the homology of a piano algebra.
We also explicitly describe the graded endomorphism ring of a class of classical generators of $\ocC$, which includes the class of limit generators.

Finally, in \Cref{Sec: The Final Equivalences}, we show that any two piano algebras $\Lambda$ and $\Gamma$ that are associated to homeomorphic unpunctured marked discs are derived equivalent.
To do this we use the equivalence of triangulated categories between $\ocC$ and $\mathrm{perf}\Lambda_n$, and the classification of classical generators of $\ocC$ \cite{Generators}.

\subsubsection*{Conventions}

Throughout, we let $k$ be an algebraically closed field.
All modules will be considered as right modules, unless otherwise stated, and we will compose paths from left to right.
That is, if $a$ and $b$ are paths, then we write $ab$ to mean the path from the source of $a$ to the target of $b$.
However, morphisms will be composed right to left, \ie if $f : X \rightarrow Y$ and $g :Y \rightarrow Z$, then $gf : X \rightarrow Z$.
For a category $\cC$ and two objects $X,Y \in \cC$, we write $\cC(X,Y) = \Hom[\cC]{X}{Y}$.

We will assume that the suspension functor $[1]_{\cT}$ for a triangulated category $\cT$ has a natural inverse, $[-1]_{\cT}$, and we will simply denote the suspension functor by $[1]$ if the category is clear from context.
We will always take triangle to mean a distinguished triangle in a triangulated category.

For a graded/differential graded algebra $A$, we denote the degree $m \in \mathbb{Z}$ component of $A$ by $A^n$.

\section{Discrete Cluster Categories}\label{Sec: Those cats}

\subsection{Discrete Cluster Categories of Dynkin Type \texorpdfstring{$A_{\infty}$}{Ainf}}\label{Sec:Discrete}

We start by defining the marked surface we use to construct the category $\cC_n$, where we always consider the closed disc $D^2$ to have an anti-clockwise orientation on the boundary $\partial D^2 = S^1$.

\begin{definition}\label[def]{Def:admiss}
A subset $\sM$ of the circle $S^1$ is called \textit{admissible} if it satisfies the following conditions:
\begin{enumerate}
    \item $\sM$ has infinitely many elements,
    \item $\sM \subset S^1$ is a discrete subset,
    \item $\sM$ satisfies the \textit{two-sided limit condition}, i.e.\ each $x \in S^1$ which is the limit of a sequence in $\sM$ is a limit of both an increasing and decreasing sequence from $\sM$ with respect to the cyclic order.
\end{enumerate}
\end{definition}

We call the points at which the two-sided limits in $\sM$ converge \textit{accumulation points}.
Note that these accumulation points are not in $\sM$, as they do not have defined successors and predecessors, meaning $\sM$ would not be discrete if they were to be included.
As in \cite{Gratz2017}, we may think of $\sM$ as the vertices of the $\infty$-gon.
For an admissible subset $\sM \in S^1$, we label the set of accumulation points $L(\sM)$, and give them a cyclic ordering induced by the orientation of $S^1$.

Let $\mathfrak{a} \in L(\sM)$ be an accumulation point, we define $\mathfrak{a}^+ \in L(\sM)$ (resp.\ $\mathfrak{a}^- \in L(\sM)$) to be the accumulation point such that $\mathfrak{a} \leq \mathfrak{b} \leq \mathfrak{a}^+$ with $\mathfrak{b} \in L(\sM)$ ($\mathfrak{a}^- \leq \mathfrak{c} \leq \mathfrak{a}$ with $\mathfrak{c} \in L(\sM)$) implies $\mathfrak{b}=\mathfrak{a}$ or $\mathfrak{b}=\mathfrak{a}^+$ (resp.\ $\mathfrak{c}=\mathfrak{a}^-$ or $\mathfrak{c}=\mathfrak{a}$).
Moreover, we say that $\mathfrak{a}$ is its own \textit{successor} and \textit{predecessor}.
Let $\osM := L(\sM) \cup \sM$.

\begin{figure}[h]
	\centering
	\begin{tikzpicture}
		\draw (0,0) circle (2cm);
		
		\foreach[count = \i] \txt in {1,...,3}
		\draw[fill=white] ({2*cos(47+120*\i)},{2*sin(47+120*\i)}) circle (0.1cm);
		
		\foreach[count = \i] \txt in {1,...,7}
		\draw[thin] ({1.95*cos(47-(30/\i))},{1.95*sin(47-(30/\i))}) -- ({2.05*cos(47-(30/\i))},{2.05*sin(47-(30/\i))});
		
		\foreach[count = \i] \txt in {1,...,7}
		\draw[thin] ({1.95*cos(47+(30/\i))},{1.95*sin(47+(30/\i))}) -- ({2.05*cos(47+(30/\i))},{2.05*sin(47+(30/\i))});
		
		\foreach[count = \i] \txt in {1,...,7}
		\draw[thin] ({1.95*cos(167-(30/\i))},{1.95*sin(167-(30/\i))}) -- ({2.05*cos(167-(30/\i))},{2.05*sin(167-(30/\i))});
		
		\foreach[count = \i] \txt in {1,...,7}
		\draw[thin] ({1.95*cos(167+(30/\i))},{1.95*sin(167+(30/\i))}) -- ({2.05*cos(167+(30/\i))},{2.05*sin(167+(30/\i))});
		
		\foreach[count = \i] \txt in {1,...,7}
		\draw[thin] ({1.95*cos(287-(30/\i))},{1.95*sin(287-(30/\i))}) -- ({2.05*cos(287-(30/\i))},{2.05*sin(287-(30/\i))});
		
		\foreach[count = \i] \txt in {1,...,7}
		\draw[thin] ({1.95*cos(287+(30/\i))},{1.95*sin(287+(30/\i))}) -- ({2.05*cos(287+(30/\i))},{2.05*sin(287+(30/\i))});
		
		\foreach[count = \i] \txt in {1,2,3}
		\draw[thin] ({1.95*cos(77+15*\i)},{1.95*sin(77+15*\i)}) -- ({2.05*cos(77+15*\i)},{2.05*sin(77+15*\i)});
		
		\foreach[count = \i] \txt in {1,2,3}
		\draw[thin] ({1.95*cos(197+15*\i)},{1.95*sin(197+15*\i)}) -- ({2.05*cos(197+15*\i)},{2.05*sin(197+15*\i)});
		
		\foreach[count = \i] \txt in {1,2,3}
		\draw[thin] ({1.95*cos(317+15*\i)},{1.95*sin(317+15*\i)}) -- ({2.05*cos(317+15*\i)},{2.05*sin(317+15*\i)});
		
		\node at ({2.3*cos(47)},{2.3*sin(47)}) {\footnotesize $\mathfrak{a}$};
		\node at ({2.3*cos(92)},{2.3*sin(92)}) {\footnotesize $x^+$};
		\node at ({2.3*cos(107)},{2.3*sin(107)}) {\footnotesize $x$};
		\node at ({2.3*cos(122)},{2.3*sin(122)}) {\footnotesize $x^-$};
	\end{tikzpicture}
	\caption{An admissible subset $\sM$ of $S^1$. The marked points in $\sM$ converge to the accumulation points represented as small circles, and each marked point $x$ has both a predecessor and a successor, labelled $x^-$ and $x^+$ respectively.}
	\label{fig:admissable}
\end{figure}
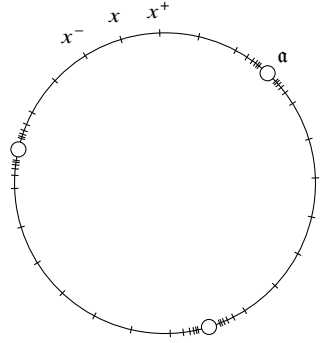

\begin{definition}\label[def]{Def:arcs}
A \textit{arc of} $\osM$ is a subset $\ell_X=\{x_0,x_1\} \subset \sM$ where $x_1 \notin \{x_0^-, x_0, x_0^+\}$, where $x^+$ and $x^-$ are the successor and predecessor respectively to $x \in \sM$.
If $\ell_Y= \{y_0,y_1\}$ is another arc, then $\ell_X$ and $\ell_Y$ \textit{cross} if $x_0<y_0<x_1<y_1 <x_0$ or $x_0<y_1<x_1<y_0<x_0$.

We will call the points $x_0$ and $x_1$ the \textit{endpoints} of $\ell_X$.
\end{definition}

We may think of an arc $\ell_X=\{x_0,x_1\}$ as being a representative of the isotopy class of non-self-intersecting curves in $D^2$ between the marked points $x_0$ and $x_1$.
Two arcs cross if any of their representative curves cross in the interior of $D^2$.

There are four types of arcs for us to consider; \textit{short arcs}, \textit{long arcs}, \textit{limit arcs}, and \textit{double limit arcs}.
An arc $\ell = \{x,y\}$ with endpoints in $\sM$ is a short arc if $\mathfrak{a} < x < y < \mathfrak{a}^+$, and is a long arc if $\mathfrak{a}^- < x < \mathfrak{a} < y < \mathfrak{a}^+$, for some $\mathfrak{a} \in L(\sM)$.
An arc $\ell = \{x,y\}$ is a limit arc if $x \in L(\sM)$ and $y \in \sM$, or vice versa, and $\ell$ is a double limit arc if $x,y \in L(\sM)$.

Note that there are no long arcs, nor double limit arcs when $n=1$.
We write \textit{(double) limit arcs} to mean the set of all limit arcs and double limit arcs.

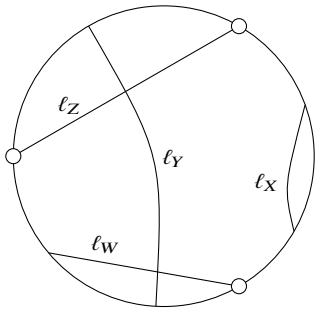
\begin{figure}[h]
    \centering
    \begin{tikzpicture}
    \draw (0,0) circle (2cm);
    \draw ({2*cos(20)},{2*sin(20)}) .. controls ({1.6*cos(355)},{1.6*sin(345)}) .. ({2*cos(330)},{2*sin(330)}) node [pos=0.5,left] {\footnotesize $\ell_X$};
    \draw ({2*cos(267)},{2*sin(267)}) .. controls (0,0) .. ({2*cos(120)},{2*sin(120)}) node [pos=0.5,right] {\footnotesize $\ell_Y$};
    \draw (-2,0) -- (1,1.73) node [pos=0.25,above] {\footnotesize $\ell_Z$};
    \draw (1,-1.73) -- ({2*cos(220)},{2*sin(220)}) node [pos=0.7,above] {\footnotesize $\ell_W$};
    \draw[fill=white] (-2,0) circle (0.1cm);
    \draw[fill=white] (1,1.73) circle (0.1cm);
    \draw[fill=white] (1,-1.73) circle (0.1cm);
    \end{tikzpicture}
    \caption{Four arcs of $\osM$, where $\ell_X$ is a short arc, $\ell_Y$ is a long arc, $\ell_W$ is a limit arc, and $\ell_Z$ is a double limit arc.}
    \label[fig]{fig:short and long arcs}
\end{figure}

In \cite{Paquette2020}, the authors construct a Hom-finite, Krull-Schmidt, $k$-linear, triangulated category, denoted $\cC(D^2,\osM)$ here, such that the indecomposable objects are in bijection to the arcs of $\osM$, for $\lvert L(\sM) \rvert =n$.
The suspension functor of $\cC(D^2,\osM)$ acts on indecomposable objects by taking endpoints in $\sM$ to their predecessor, and taking an endpoint in $\osM$ to itself.
The authors of \cite{Paquette2020} show that $\cC(D^2,\osM)(X,Y) \cong k$ if and only if $\ell_X$ and $\ell_{Y[-1]}$ cross for $X,Y$ indecomposable objects in $\cC(D^2,\osM)$, and $\cC(D^2,\osM)(X,Y) =0$ otherwise.

These categories are known as the \textit{Paquette-Y\i ld\i r\i m completions of discrete cluster categories of Dynkin type} $A_{\infty}$, however we shall refer to them as \textit{Paquette-Y\i ld\i r\i m categories} for the sake of brevity.

Let $\sM$ and $\mathscr{N}$ be admissible subsets of $S^1$ such that $\lvert L(\sM) \rvert = \lvert  L(\overline{\mathscr{N}}) \rvert = n$, then there is an equivalence of categories between $\cC(D^2,\osM)$ and $\cC(D^2,\mathscr{N})$.
Therefore for all $n \in \mathbb{Z}_{\geq 1}$, we will consider an admissible subset of $S^1$ with $n$ accumulation points, $\sM_n$, and so we consider the category $\ocC := \cC(S^1,\osM[n])$ as the representative of the equivalence class of discrete cluster categories of Dynkin type $A_{\infty}$ with $n$ accumulation points.

Let $X = \{x_0,x_1\}$ and $Y = \{y_0,y_1\}$ be indecomposable objects in $\ocC$ such that $y_0 < x_0 < y_1 < x_1$, then there exists two triangles in $\ocC$;
\begin{align*}
	X \rightarrow A \oplus B & \rightarrow Y \rightarrow X[1],\\
	Y \rightarrow C \oplus D & \rightarrow X \rightarrow Y[1].
\end{align*}
where $A = \{y_1,x_1\}$, $B = \{y_0,x_0\}$, $C = \{y_0,x_1\}$, and $D = \{x_0,y_1\}$.
Moreover, for indecomposable objects $U =\{\mathfrak{a},u\}$ and $Y = \{ \mathfrak{a},v\}$ in $\ocC$ with $\mathfrak{a} < u < v < \mathfrak{a}$, then there also exists a triangle in $\ocC$ of the form;
\[
X \rightarrow Z \rightarrow Y \rightarrow X[1],
\]
where $Z = \{u,v\}$.

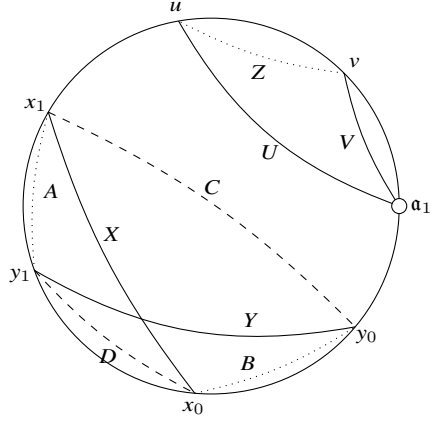
\begin{figure}[H]
	\centering
	\begin{tikzpicture}
		\draw (0,0) circle (4cm);
		\draw[thin] (4,0) to [bend left =10] node [pos=0.5, left] {\footnotesize $V$} ({4*cos(45)},{4*sin(45)});
		\draw[thin] (4,0) to [bend left = 20] node [pos=0.5,below] {\footnotesize $U$} ({4*cos(100)},{4*sin(100)});
		
		\draw[dotted, thin] ({4*cos(45)},{4*sin(45)}) to [bend left =10] node [pos=0.5, below] {\footnotesize $Z$} ({4*cos(100)},{4*sin(100)});
		
		\draw[thin] ({4*cos(200)},{4*sin(200)}) to [bend right =20] node [pos=0.7,above] {\footnotesize $Y$} ({4*cos(320)},{4*sin(320)});
		\draw[thin] ({4*cos(150)},{4*sin(150)}) to [bend right =10] node [pos=0.4,right] {\footnotesize $X$} ({4*cos(265)},{4*sin(265)});
		
		\draw[dashed, thin] ({4*cos(200)},{4*sin(200)}) to [bend right =10] node [pos=0.5,below] {\footnotesize $D$} ({4*cos(265)},{4*sin(265)});
		\draw[dashed, thin] ({4*cos(150)},{4*sin(150)}) to [bend left =10] node [pos=0.5,above] {\footnotesize $C$} ({4*cos(320)},{4*sin(320)});
		
		\draw[dotted, thin] ({4*cos(200)},{4*sin(200)}) to [bend left =10] node [pos=0.5,right] {\footnotesize $A$} ({4*cos(150)},{4*sin(150)});
		\draw[dotted, thin] ({4*cos(265)},{4*sin(265)}) to [bend right =10] node [pos=0.3,above] {\footnotesize $B$} ({4*cos(320)},{4*sin(320)});
		
		\draw[fill=white] (4,0) circle (0.1cm);
		
		\node at (4.3,0) {\footnotesize $\mathfrak{a}_1$};
		\node at ({4.2*cos(45)},{4.2*sin(45)}) {\footnotesize $v$};
		\node at ({4.2*cos(100)},{4.2*sin(100)}) {\footnotesize $u$};
		\node at ({4.2*cos(320)},{4.2*sin(320)}) {\footnotesize $y_0$};
		\node at ({4.2*cos(200)},{4.2*sin(200)}) {\footnotesize $y_1$};
		\node at ({4.2*cos(265)},{4.2*sin(265)}) {\footnotesize $x_0$};
		\node at ({4.2*cos(150)},{4.2*sin(150)}) {\footnotesize $x_1$};
		
	\end{tikzpicture}
	\caption{The arcs corresponding to triangles coming from non-split extensions between indecomposable objects.}
\end{figure}
\begin{definition}\label[def]{Def:Orbits}
Let $X$ be an indecomposable object in $\cC_n$.
Let $\sM_X \subseteq \sM$ be the set of marked points that are the endpoints of the arcs corresponding to suspensions and desuspensions of $X$.

If $A \cong \bigoplus^l_{i=1} X_i$, with all $X_i$ indecomposable, then $\sM_A = \bigcup_{i=1}^l \sM_{X_i}$.
We call $\sM_A$ the \textit{orbit of} $A$ \textit{in} $\sM$.
If we have $\sM_A=\sM$, then we say $\ell_A$ has a \textit{complete orbit in} $\sM$.
\end{definition}

In particular, $\sM_X$ is equal to the union of the segments containing an endpoint of $\ell_X$, and so any object in $\cC_1$ has a complete orbit in $\sM$.

The following lemma shows how the $\mathrm{Ext}^1$-spaces are given in $\ocC$.

\begin{proposition}\cite[Proposition 3.14]{Paquette2020}\label[prop]{Prop:PY3.14}
Let $X,Y \in \ocC$ be indecomposable objects.
Then $\ocC(X,Y\wb{1})$ is at most one dimensional.
It is one dimensional if and only if one of the following conditions are met for the arcs $\ell_X$ and $\ell_Y$:
\begin{itemize}
    \item $\ell_X, \ell_Y$ cross,
    \item $\ell_X \neq \ell_Y$ share exactly one accumulation point, and we can go from $\ell_X$ to $\ell_Y$ by rotating $\ell_X$ about their common endpoint following the orientation of $S^1$,
    \item $\ell_X = \ell_Y$ are double limit arcs.
\end{itemize}
\end{proposition}.

We shall draw no distinction between indecomposable objects in $\ocC$ and the corresponding arcs, meaning that from now on $X$ will represent both the indecomposable object in $\ocC$ and the arc in $(S^1,\overline{\mathscr{M}})$, previously denoted by $\ell_X$.

\begin{lemma}\label[lem]{Lem: Factoring in Cn}
	Let $Y,Z \in \ocC$ be non-isomorphic indecomposable objects such that $\ocC(Y,Z) \cong k$, and let $Y = \{y_1,y_2\}$ and $Z = \{z_1,z_2\}$.
	Then any morphism $f: Y \rightarrow Z$ factors through an indecomposable object $W = \{w_1,w_2\}$ if and only if $z_1 \geq w_1 \geq y_1$ and $z_2 \geq w_2 \geq y_2$ with respect to $(S^1,\osM)$.
\end{lemma}

\begin{proof}
	This is a direct consequence of \cite[Lemma 2.4.2]{Igusa2013}, and \cite[Lemmas 3.8, 3.9]{Paquette2020}.
\end{proof}

Let $X,Y \in \ocC$ be indecomposable objects, suppose that $X = \{x_1,x_2\}$ and $Y= \{y_1,y_2\}$, and there exists a non-zero morphism $f: X \rightarrow Y$.
If we have 
\[
a_1 < x_1 \leq y_1 < x_2 \leq y_2 \leq a_1,
\]
such that $x_2 \neq a_1$, then we call $f$ a \textit{forward morphism}.
Alternatively, if we have 
\[
a_1 \leq y_1 < x_1 \leq y_2 < x_2 \leq a_1,
\]
then we call $f$ a \textit{backwards morphism}.

The following is a direct consequence of \Cref{Lem: Factoring in Cn}.

\begin{corollary}\label[cor]{Cor: Forwards Backwards Composition}
	Let $f: X \rightarrow Y$ and $g: Y \rightarrow Z$ be non-zero morphisms between indecomposable objects in $\ocC$, then;
	\begin{enumerate}
		\item if both $f$ and $g$ are forward morphisms, then $gf$ is a non-zero forward morphism,
		\item if one of $f$ or $g$ is forward and the other backwards, then $gf$ is a non-zero backwards morphism,
		\item if both $f$ and $g$ are backwards morphisms, then $gf=0$.
	\end{enumerate}
\end{corollary}

\subsection{Classical Generators of \texorpdfstring{$\ocC$}{Cn}}

In a triangulated category $\cT$, a subcategory $\cS$ is called \textit{thick} if it is a triangulated subcategory closed under direct summands.
Let $G \in \cT$, then we denote by $\langle G \rangle$ the smallest thick subcategory of $\cT$ containing $G$.
The object $G$ is a \textit{classical generator} of $\cT$ if $\langle G \rangle = \cT$, and $G$ is \textit{minimal} if no direct summand of $G$ is a classical generator.

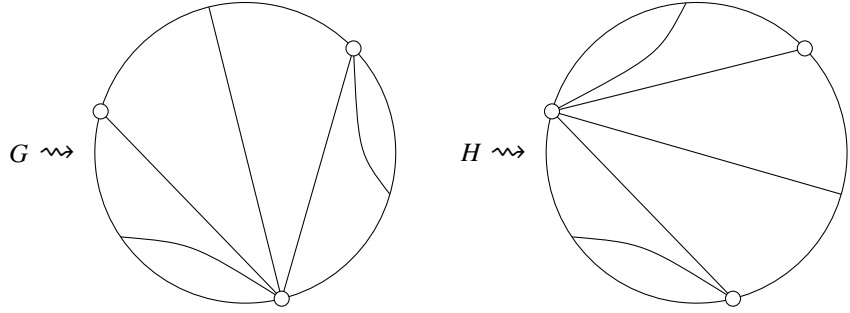
\begin{figure}[H]
	\centering
	\begin{tikzpicture}
		\draw (-3,0) circle (2cm);
		\draw (3,0) circle (2cm);
		\draw[thin] ({-3+2*cos(284)},{2*sin(284)}) -- ({-3+2*cos(164)},{2*sin(164)});
		\draw[thin] ({-3+2*cos(284)},{2*sin(284)}) -- ({-3+2*cos(44)},{2*sin(44)});
		\draw[thin] ({-3+2*cos(284)},{2*sin(284)}) -- ({-3+2*cos(104)},{2*sin(104)});
		\draw[thin] ({-3+2*cos(284)},{2*sin(284)}) .. controls (-3.7,-1.2) ..  ({-3+2*cos(214)},{2*sin(214)});
		\draw[thin] ({-3+2*cos(44)},{2*sin(44)}) .. controls (-1.5,0) .. ({-3+2*cos(344)},{2*sin(344)});
		\draw[thin] ({3+2*cos(284)},{2*sin(284)}) -- ({3+2*cos(164)},{2*sin(164)});
		\draw[thin] ({3+2*cos(164)},{2*sin(164)}) -- ({3+2*cos(44)},{2*sin(44)});
		\draw[thin] ({3+2*cos(164)},{2*sin(164)}) -- ({3+2*cos(344)},{2*sin(344)});
		\draw[thin] ({3+2*cos(284)},{2*sin(284)}) .. controls (2.3,-1.2) .. ({3+2*cos(214)},{2*sin(214)});
		\draw[thin] ({3+2*cos(164)},{2*sin(164)}) .. controls (2.5,1.3) .. ({3+2*cos(94)},{2*sin(94)});
		\draw[fill=white] ({-3+2*cos(284)},{2*sin(284)}) circle (0.1cm);
		\draw[fill=white] ({3+2*cos(284)},{2*sin(284)}) circle (0.1cm);
		\draw[fill=white] ({-3+2*cos(44)},{2*sin(44)}) circle (0.1cm);
		\draw[fill=white] ({3+2*cos(44)},{2*sin(44)}) circle (0.1cm);
		\draw[fill=white] ({-3+2*cos(164)},{2*sin(164)}) circle (0.1cm);
		\draw[fill=white] ({3+2*cos(164)},{2*sin(164)}) circle (0.1cm);
		\node at (-6,0) {$G$};
		\node at (0,0) {$H$};
		\node at (-5.5,0) {\Large $\rightsquigarrow$};
		\node at (0.5,0) {\Large $\rightsquigarrow$};
	\end{tikzpicture}
	\caption{The objects $G$ and $H$ are an example of two generators such that $H$ and $G$ are in the same equivalence class.}
	\label[fig]{fig:enter-label}
\end{figure}

The classical generators of $\ocC$ were classified in \cite{Generators}, using a notion of an object being homologically connected.
An object $X$ in a triangulated category $\cT$ is \textit{homologically connected}, if for any two indecomposable objects $Y,Z \in \langle X \rangle$, there exists a sequence of non-zero, irreducible morphisms between indecomposables in $\langle X \rangle$, starting and ending at $Y$ and $Z$.
Note that there is no requirement on the direction or composition of the morphisms.

A similar, yet subtly different, notion of \textit{Ext-connected} was studied in \cite{PageThick}.

\begin{theorem}[\cite{Generators}]\label{Thm:Gens}
	Let $G$ be an object in $\ocC$. Then $G$ is a classical generator if and only if $G$ is homologically connected and $G$ has a complete orbit in $\mathscr{M}$.
\end{theorem}

From now on, when we refer to generators we mean classical generators rather than any other definition of generator in a triangulated category.

We may impose an equivalence relation on the class of classical generators in $\ocC$, where $G \sim G'$ if there exists the an additive equivalence of categories between the categories,
\[
\mathrm{add}(G[i] \mid \forall i \in \mathbb{Z})  \xrightarrow{\sim} \mathrm{add}(G'[i] \mid \forall i \in \mathbb{Z}).
\]

\begin{definition}
	Let $G \in \ocC$ be a classical generator.
	We say that $G$ is a \textit{limit generator} if all indecomposable direct summands of $G$ are (double) limit arcs, and for any two indecomposable direct summands $G_1,G_2$ of $G$, then $G_1$ and $G_2[i]$ do not cross for any $i \in \mathbb{Z}$.
\end{definition}

\begin{proposition}\label[prop]{Prop: Ind triangles everywhere}
	Let $G$ be a limit generator of $\ocC$.
	Then for each indecomposable object $X \in \ocC$, there exists a triangle in $\ocC$,
	\[
	Y \rightarrow X \rightarrow Z \rightarrow Y[1],
	\]
	with $Y,Z$ indecomposable, where $Y \in \lang[a]{G}$, $Z \in \lang[b]{G}$, and $X \in \lang[a+b]{G}$.
\end{proposition}

\begin{proof}
	Suppose that $X$ is either a long arc or a short arc.
	Then there exists a triangle of the form,
	\[
	Y \rightarrow X \rightarrow Z \rightarrow Y[1],
	\]
	where $Y,Z$ are limit arcs \cite{Paquette2020}, and hence are indecomposable.
	There also exists a choice such that at least one of $Y$ or $Z$ is contained in $\lang[1]{G}$.
	It is left to show that we may obtain limit arcs in the same manner.
	
	Now suppose that $X$ is a limit arc, then $X$ shares an endpoint with two indecomposable objects in $\lang[1]{G}$, by \Cref{Thm:Gens}.
	Call those indecomposable objects $W_0$ and $W_m$.
	Also by \Cref{Thm:Gens}, there exists a sequence of indecomposable objects $W_0,W_1,W_2, \ldots, W_m \in \lang[1]{G}$, such that $W_j$ and $W_{j+1}$ share an endpoint, but $W_j$ and $W_{j+2}$ do not share an endpoint for all $j=0,\ldots,m-1$.
	This is a direct consequence of $G$ being homologically connected.
	Therefore we have a sequence of triangles with all terms indecomposable by \cite{Paquette2020};
	\[
	Y_j \rightarrow X_j \rightarrow Z_j \rightarrow Y_j[1]
	\]
	where $Y_j \neq Z_j$, $Y_j, Z_j \in \{X_{j-1},W_{j+1}\}$ for $j =1,\ldots, m-1$, $X_{m-1} \cong X$, and $Y_0,Z_0 \in \{W_0,W_1\}$.
	Hence we have found such a triangle for a limit arc $X$, and we have proven our claim.
\end{proof}

\section{Dissections of Marked Surfaces}\label{Sec: Dissections}

\subsection{Admissible Dissections of Marked Surfaces}

We recall the notions of admissible dissections studied in \cite{AmiotPlamondonSchroll,OpperPlamondonSchroll}.

\begin{definition}\cite{AmiotPlamondonSchroll}
	A \textit{marked surface} is a triple $(\Sigma,M,P)$, where 
	\begin{itemize}
		\item $\Sigma$ is an orientated open smooth surface, with boundary denoted by $\partial \Sigma$;
		\item $M = M_{\rcirc} \cup M_{\gbullet}$, is a finite set of marked points on $\partial \Sigma$. We denote the elements of $M_{\rcirc}$ and $M_{\gbullet}$ by $\rcirc$ and $\gbullet$ respectively.
		They are required to alternate on each connected component of $\partial \Sigma$, and each such component is required to have at least one marked point;
		\item $P = P_{\rcirc} \cup P_{\gbullet}$ is a finite set of marked points in $\Sigma \backslash \partial \Sigma$, called \textit{punctures}, and we denote the elements of $P_{\rcirc}$ and $P_{\gbullet}$ by $\rcirc$ and $\gbullet$ respectively.
	\end{itemize}
	If the surface has empty boundary, we require that both $P_{\rcirc}$ and $P_{\gbullet}$ are non-empty.
\end{definition}

\begin{definition}
	A $\rcirc$-\textit{arc} (resp.\ $\gbullet$-\textit{arc}) is a smooth map $\gamma$ from the open interval $(0,1)$ to $\Sigma \backslash P$ such that its \textit{endpoints} $\lim_{x \rightarrow 0} \gamma (x)$ and $\lim_{x \rightarrow 1} \gamma (x)$ are in $M_{\rcirc} \cup P_{\rcirc}$ (resp.\ $M_{\gbullet} \cup P_{\gbullet}$).
	We require the curve not be contractible to a point in $M_{\rcirc} \cup P_{\rcirc}$ (resp.\ $M_{\gbullet} \cup P_{\gbullet}$).
\end{definition}

We consider arcs up to isotopy, and say that two arcs are intersecting if any choice of their homotopic representatives intersect.

\begin{definition}
	A collection of pairwise non-intersecting and different $\rcirc$-arcs $\{\gamma_1, \ldots, \gamma_r\}$ on the surface $(\Sigma,M,P)$ is \textit{admissible} if the arcs $\gamma_1,\ldots,\gamma_r$ do not enclose a subsurface containing non punctures of $P_{\gbullet}$ and with no boundary segment on its boundary.
	A maximal admissible collection of $\rcirc$-arcs is an \textit{admissible} $\rcirc$-\textit{dissection}.
	
	An \textit{admissible} $\gbullet$-\textit{dissection} is defined analogously.
\end{definition}

\begin{proposition}\cite[proposition 1.11]{AmiotPlamondonSchroll}\label[prop]{Prop: arcs in admissible}
Let $(\Sigma,M,P)$ be a marked surface, with $\Sigma$ a surface of genus $g$ such that $\partial \Sigma$ has $b$ connected components.
Then an admissible collection of $\rcirc$-arcs is an admissible $\rcirc$-dissection if and only if it contains exactly $\lvert M_{\rcirc} \rvert + \lvert P \rvert + b + 2g -2$ arcs.
\end{proposition}

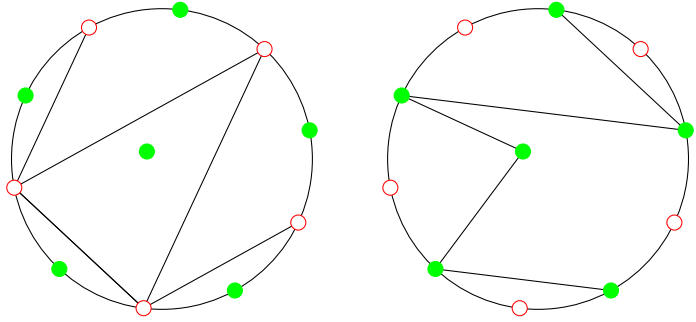
\begin{figure}[H]
	\centering
	\begin{tikzpicture}
		\draw[thin] ({2*sin(43)-2.5},{2*cos(43)}) -- ({2*sin(43+72*2)-2.5},{2*cos(43+72*2)}) -- ({2*sin(43+72*3)-2.5},{2*cos(43+72*3)}) -- ({2*sin(43)-2.5},{2*cos(43)});
		\draw[thin] ({2*sin(43+72)-2.5},{2*cos(43+72)}) -- ({2*sin(43+72*2)-2.5},{2*cos(43+72*2)}) -- ({2*sin(43+72*3)-2.5},{2*cos(43+72*3)}) -- ({2*sin(43+72*4)-2.5},{2*cos(43+72*4)});
		
		\draw[thin] ({2*sin(7+72*2)+2.5},{2*cos(7+72*2)})-- ({2*sin(7+72*3)+2.5},{2*cos(7+72*3)}) -- (2.3,0.1) -- ({2*sin(7+72*4)+2.5},{2*cos(7+72*4)}) -- ({2*sin(7+72)+2.5},{2*cos(7+72)}) -- ({2*sin(7)+2.5},{2*cos(7)});
		
		\draw (-2.5,0) circle (2cm);
		\foreach [count = \i] \txt in {1,...,5}
		\draw[thin,red,fill=white] ({2*sin(43+72*\i)-2.5},{2*cos(43+72*\i)}) circle (0.1cm);
		
		\foreach [count = \i] \txt in {1,...,5}
		\draw[thin,green,fill=green] ({2*sin(7+72*\i)-2.5},{2*cos(7+72*\i)}) circle (0.1cm);
		
		\draw[thin,green,fill=green] (-2.7,0.1) circle (0.1cm);
		
		\draw (2.5,0) circle (2cm);
		\foreach [count = \i] \txt in {1,...,5}
		\draw[thin,red,fill=white] ({2*sin(43+72*\i)+2.5},{2*cos(43+72*\i)}) circle (0.1cm);
		
		\foreach [count = \i] \txt in {1,...,5}
		\draw[thin,green,fill=green] ({2*sin(7+72*\i)+2.5},{2*cos(7+72*\i)}) circle (0.1cm);
		
		\draw[thin,green,fill=green] (2.3,0.1) circle (0.1cm);
	\end{tikzpicture}
	\caption{An admissible $\rcirc$-dissection and an admissible $\gbullet$-dissection of the marked surface $(D^2,M,P)$.}
	\label[fig]{Fig: Admissible dissections}
\end{figure}

\subsection{The Locally Gentle Algebra of an Admissible Dissection}

Gentle algebras are a class of finite-dimensional algebras that have been thoroughly studied and have a well understood representation theory.
They may be described in a particularly nice way in terms of a path algebra with relations.

\begin{definition}\label[def]{Def: Locally gentle}
	An algebra $A$ is \textit{locally gentle} if it is isomorphic to an algebra $kQ/I$ such that;
	\begin{itemize}
		\item $Q$ has finitely many vertices, is connected, and at most 2 arrows beginning and 2 arrows ending at each vertex;
		\item $I$ is an admissible ideal of $Q$ (that is, $R^m \subset I \subset R^2$, where $R$ is the ideal generated by arrows, and $2 \leq m \in \mathbb{Z}$);
		\item $I$ is generated by paths of length 2;
		\item for every arrow $\alpha$ of $Q$, there is at most one $\beta$ and one arrow $\gamma$ such that $\beta \alpha, \alpha \gamma \in I$; and there is at most one $\beta'$ and one arrow $\gamma'$ such that $\beta' \alpha, \alpha \gamma' \not\in I$.
	\end{itemize}
	We say that $A$ is \textit{gentle} if it is a finite-dimensional locally gentle algebra.
\end{definition}

By an abuse of naming convention, we may call the pair $(Q,I)$ a \textit{gentle quiver} if $kQ/I$ is a gentle algebra.

We recall from \cite{OpperPlamondonSchroll} how we may construct a (locally) gentle algebra from an admissible $\rcirc$-dissection.

\begin{definition}\label[def]{Def: Quiver from dissection}
	Let $\Omega$ be an admissible $\rcirc$-dissection of a marked surface $(\Sigma,M,P)$.
	The $k$-algebra $A(\Omega)$ is the quotient of the path algebra of the quiver $\widehat{Q}(\Omega)$ by the ideal $\widehat{I}(\Omega)$ defined as follows:
	\begin{enumerate}
		\item the vertices of $\widehat{Q}(\Omega)$ are in bijection with the $\rcirc$-arcs.
		\item there is an arrow $i \rightarrow j$ in $\widehat{Q}(\Omega)$ whenever the $\rcirc$-arcs $i$ and $j$ meet at a marked point $\rcirc$, with $i$ immediately preceding $j$ in the anticlockwise order around $\rcirc$.
		\item the ideal $\widehat{I}(\Omega)$ is generated by the following relations: whenever $i$ and $j$ meet at a marked point as above, the other end of $j$ meets $l$ at a marked point as above, then the composition of the corresponding arrows $i \rightarrow j$ and $j \rightarrow l$ is a relation.
	\end{enumerate}
\end{definition}

\begin{figure}[H]
	\centering
	\begin{tikzpicture}
		\draw[thin] ({2*sin(43)-2.5},{2*cos(43)}) -- ({2*sin(43+72*2)-2.5},{2*cos(43+72*2)})node[pos=0.5] (1) {$\empty$};
		\draw[thin] ({2*sin(43+72*2)-2.5},{2*cos(43+72*2)}) -- ({2*sin(43+72*3)-2.5},{2*cos(43+72*3)})node[pos=0.5] (2) {$\empty$};
		\draw[thin] ({2*sin(43+72*3)-2.5},{2*cos(43+72*3)})  -- ({2*sin(43)-2.5},{2*cos(43)})node[pos=0.5] (3) {$\empty$};
		\draw[thin] ({2*sin(43+72)-2.5},{2*cos(43+72)}) -- ({2*sin(43+72*2)-2.5},{2*cos(43+72*2)})node[pos=0.5] (4) {$\empty$};
		\draw[thin] ({2*sin(43+72*3)-2.5},{2*cos(43+72*3)})  -- ({2*sin(43+72*4)-2.5},{2*cos(43+72*4)})node[pos=0.5] (5) {$\empty$};
		
		\draw[thin] ({2*sin(7+72*2)+2.5},{2*cos(7+72*2)})-- ({2*sin(7+72*3)+2.5},{2*cos(7+72*3)}) node[pos=0.5] (A) {$\empty$};
		\draw[thin] ({2*sin(7+72*3)+2.5},{2*cos(7+72*3)}) -- (2.6,-0.25) node[pos=0.7] (B) {$\empty$};
		\draw[thin] (2.6,-0.25) -- ({2*sin(7+72*4)+2.5},{2*cos(7+72*4)}) node[pos=0.3] (C) {$\empty$};
		\draw[thin] ({2*sin(7+72*4)+2.5},{2*cos(7+72*4)}) -- ({2*sin(7+72)+2.5},{2*cos(7+72)}) node[pos=0.5] (D) {$\empty$};
		\draw[thin] ({2*sin(7+72)+2.5},{2*cos(7+72)}) -- ({2*sin(7)+2.5},{2*cos(7)}) node[pos=0.5] (E) {$\empty$};
		
		\draw[thin,red,dashed,->] (1)-- (2) node[pos=0.5] (1') {\empty} node[pos=0.25] (1'') {\empty};
		\draw[thin,red,dashed,->] (4)-- (1);
		\draw[thin,red,dashed,->] (3)-- (1) node[pos=0.75] (2') {\empty} node[pos=0.25] (2'') {\empty};
		\draw[thin,red,dashed,->] (3)-- (5);
		\draw[thin,red,dashed,->] (2)-- (3) node[pos=0.75] (3') {\empty} node[pos=0.5] (3'') {\empty};
		
		\draw[thin,red,dotted] (1'') -- (2');
		\draw[thin,red,dotted] (2'') -- (3');
		\draw[thin,red,dotted] (3'') -- (1');
		
		\draw[thin,red,dashed,->] (A) -- (B) node[pos=0.3] (A') {\empty};
		\draw[thin,red,dashed,->] (C) -- (B);
		\draw[thin,red,dashed,->] (C) -- (D) node[pos=0.8] (B') {\empty};
		\draw[thin,red,dashed,->] (E) -- (D);
		\draw[thin,red,dashed,->,looseness=4] (B) to[out=350,in=10] (C) node[pos=0.75] (C') {\empty} node[pos=0.25] (D') {\empty};
		
		\draw[thin,red,dotted] (A') -- (2.8,-0.65);
		\draw[thin,red,dotted] (2.8,0.1) -- (B');
		
		\draw (-2.5,0) circle (2cm);
		\foreach [count = \i] \txt in {1,...,5}
		\draw[thin,red,fill=white] ({2*sin(43+72*\i)-2.5},{2*cos(43+72*\i)}) circle (0.1cm);
		
		\foreach [count = \i] \txt in {1,...,5}
		\draw[thin,green,fill=green] ({2*sin(7+72*\i)-2.5},{2*cos(7+72*\i)}) circle (0.1cm);
		
		\draw[thin,green,fill=green] (-2.6,-0.25) circle (0.1cm);
		
		\draw (2.5,0) circle (2cm);
		\foreach [count = \i] \txt in {1,...,5}
		\draw[thin,red,fill=white] ({2*sin(43+72*\i)+2.5},{2*cos(43+72*\i)}) circle (0.1cm);
		
		\foreach [count = \i] \txt in {1,...,5}
		\draw[thin,green,fill=green] ({2*sin(7+72*\i)+2.5},{2*cos(7+72*\i)}) circle (0.1cm);
		
		\draw[thin,green,fill=green] (2.6,-0.25) circle (0.1cm);
	\end{tikzpicture}
	\caption{The (locally) gentle quivers associated to the admissible dissections from \Cref{Fig: Admissible dissections}.}
	\label[fig]{Fig: Quivers from Admissible}
\end{figure}
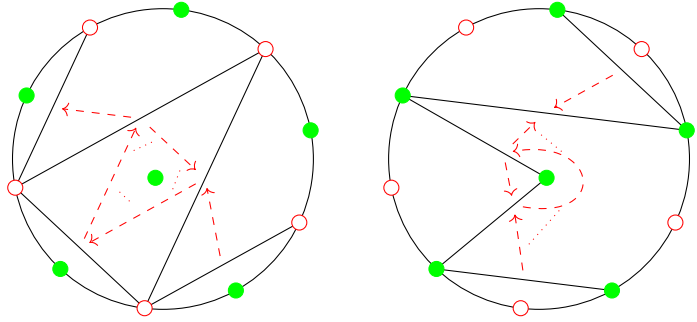

We state the following result in the language of \cite{AmiotPlamondonSchroll}.

\begin{theorem}\cite[Proposition 1.21]{OpperPlamondonSchroll} \cite[Theorem 4.10]{PPP2} \label{Thm: Gentle to Admissible}
	The assignment $((\Sigma,M,P),\Omega) \mapsto A(\Omega)$ defines a bijection between the set of homeomorphism classes of marked surfaces $(\Sigma,M,P)$ with an admissible $\rcirc$-dissection to the set of isomorphism classes of locally gentle algebras.
	Under this bijection, gentle algebras correspond to the case where $P_{\rcirc} = \emptyset$.
\end{theorem}

\subsection{Extended Admissible Dissections}

For our intended goals, admissible dissections are not enough data to understand the limit generators of $\ocC$.
Therefore, we introduce the notion of an extended admissible dissection, which imposes an additional choice upon an admissible dissection.

\begin{definition}\label[def]{Def: Ext Admissible}
	A \textit{binding arc} is a smooth map $\gamma$ from the open interval $(0,1)$ to $\Sigma \backslash P$ such that its endpoints $\lim_{x \rightarrow 0} \gamma (x) \in  M_{\rcirc} \cup P_{\rcirc}$ and $\lim_{x \rightarrow 1} \gamma (x) \in M_{\gbullet} \cup P_{\gbullet}$, respectively (or vice versa).
	
	A collection of pairwise non-intersecting and pairwise different arcs $\{\gamma_1,\ldots,\gamma_r, \delta_1,\ldots,\delta_s\}$ is an \textit{extended admissible} $\rcirc$-\textit{dissection} of $(\Sigma, M, P)$ if the collection $\{\gamma_1,\ldots,\gamma_r\}$ forms an admissible $\rcirc$-dissection, $s = \lvert M_{\gbullet} \cup P_{\gbullet} \rvert$, and $\{\delta_1,\ldots,\delta_s\}$ is a collection of binding arcs such that no two binding arcs share a $\gbullet$-endpoint.
	
	An \textit{extended admissible} $\gbullet$-\textit{dissection} is defined analogously.
\end{definition}

We find the following result immediately from the definition of an extended admissible dissection and \Cref{Prop: arcs in admissible}.

\begin{corollary}\label[cor]{Cor: Arcs in Extended}
	Let $(\Sigma,M,P)$ be a surface as in \Cref{Prop: arcs in admissible}.
	Then an extended admissible $\rcirc$-dissection has exactly $\lvert M \rvert + \lvert P \rvert + \lvert P_{\gbullet} \rvert + b + 2g -2$ arcs.
\end{corollary}

\begin{proof}
	This follows directly from \Cref{Prop: arcs in admissible} and the definition of an extended admissible $\rcirc$-dissection having a binding arc for each marked point in $M_{\gbullet} \cup P_{\gbullet}$.
\end{proof}

\begin{figure}[H]
	\centering
	\begin{tikzpicture}
		\draw[thin] ({2*sin(43)-2.5},{2*cos(43)}) -- ({2*sin(43+72*2)-2.5},{2*cos(43+72*2)}) -- ({2*sin(43+72*3)-2.5},{2*cos(43+72*3)}) -- ({2*sin(43)-2.5},{2*cos(43)});
		\draw[thin] ({2*sin(43+72)-2.5},{2*cos(43+72)}) -- ({2*sin(43+72*2)-2.5},{2*cos(43+72*2)}) -- ({2*sin(43+72*3)-2.5},{2*cos(43+72*3)}) -- ({2*sin(43+72*4)-2.5},{2*cos(43+72*4)});
		
		\draw[thin] ({2*sin(7+72*2)+2.5},{2*cos(7+72*2)})-- ({2*sin(7+72*3)+2.5},{2*cos(7+72*3)}) -- (2.3,0.1) -- ({2*sin(7+72*4)+2.5},{2*cos(7+72*4)}) -- ({2*sin(7+72)+2.5},{2*cos(7+72)}) -- ({2*sin(7)+2.5},{2*cos(7)});
		
		\draw[dashed,thin] ({2*sin(43)-2.5},{2*cos(43)}) -- (-2.7,0.1);
		\draw[dashed,thin] ({2*sin(43)-2.5},{2*cos(43)}) -- ({2*sin(7)-2.5},{2*cos(7)}); 
		\draw[dashed,thin] ({2*sin(43+72*2)-2.5},{2*cos(43+72*2)}) -- ({2*sin(7+72)-2.5},{2*cos(7+72)});
		\draw[dashed,thin] ({2*sin(43+72*2)-2.5},{2*cos(43+72*2)}) -- ({2*sin(7+72*2)-2.5},{2*cos(7+72*2)});
		\draw[dashed,thin] ({2*sin(43+72*3)-2.5},{2*cos(43+72*3)}) -- ({2*sin(7+72*3)-2.5},{2*cos(7+72*3)});
		\draw[dashed,thin] ({2*sin(43+72*3)-2.5},{2*cos(43+72*3)}) -- ({2*sin(7+72*4)-2.5},{2*cos(7+72*4)});

		\draw[dashed,thin] ({2*sin(43+72)+2.5},{2*cos(43+72)}) -- ({2*sin(7+72*3)+2.5},{2*cos(7+72*3)});
		\draw[dashed,thin] ({2*sin(43+72*2)+2.5},{2*cos(43+72*2)}) -- ({2*sin(7+72*2)+2.5},{2*cos(7+72*2)});
		\draw[dashed,thin] ({2*sin(43+72*3)+2.5},{2*cos(43+72*3)}) -- ({2*sin(7+72*4)+2.5},{2*cos(7+72*4)});
		\draw[dashed,thin] ({2*sin(43+72*4)+2.5},{2*cos(43+72*4)}) -- ({2*sin(7+72)+2.5},{2*cos(7+72)});
		\draw[dashed,thin] ({2*sin(43)+2.5},{2*cos(43)}) -- ({2*sin(7)+2.5},{2*cos(7)});
		
		\draw (-2.5,0) circle (2cm);
		\foreach [count = \i] \txt in {1,...,5}
		\draw[thin,red,fill=white] ({2*sin(43+72*\i)-2.5},{2*cos(43+72*\i)}) circle (0.1cm);
		
		\foreach [count = \i] \txt in {1,...,5}
		\draw[thin,green,fill=green] ({2*sin(7+72*\i)-2.5},{2*cos(7+72*\i)}) circle (0.1cm);
		
		\draw[thin,green,fill=green] (-2.7,0.1) circle (0.1cm);
		
		\draw (2.5,0) circle (2cm);
		\foreach [count = \i] \txt in {1,...,5}
		\draw[thin,red,fill=white] ({2*sin(43+72*\i)+2.5},{2*cos(43+72*\i)}) circle (0.1cm);
		
		\foreach [count = \i] \txt in {1,...,5}
		\draw[thin,green,fill=green] ({2*sin(7+72*\i)+2.5},{2*cos(7+72*\i)}) circle (0.1cm);
		
		\draw[thin,green,fill=green] (2.3,0.1) circle (0.1cm);
	\end{tikzpicture}
	\caption{An extended admissible $\rcirc$-dissection and an extended admissible $\gbullet$-dissection of the marked surface $(D^2,M,P)$.}
	\label[fig]{Fig: Extended admissible}
\end{figure}
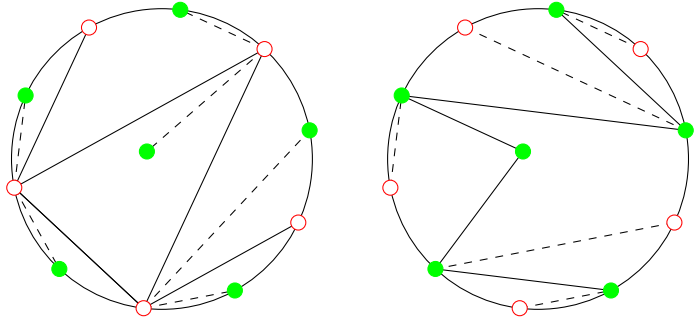

Many of the results concerning admissible dissections in \cite{AmiotPlamondonSchroll} also apply to extended admissible dissections.
This is unsurprising in view of the following result, allowing us to view extended admissible dissections as admissible dissections.

\begin{lemma}\label[lem]{Lem: extended become admissible}
	Let $\Delta$ be an extended admissible $\rcirc$-dissection of a surface $(\Sigma,M,P)$.
	Then there exists some marked surface $(\Sigma, \widetilde{M}, \widetilde{P})$, with $\widetilde{M}_{\rcirc} = M$ and $\widetilde{P}_{\rcirc} = P$, such that $\Delta$ induces an admissible $\rcirc$-dissection, $\widetilde{\Delta}$, of $(\Sigma, \widetilde{M}, \widetilde{P})$.
	
	Moreover, $\widetilde{\Delta}$ is unique up to homeomorphism.
\end{lemma}

\begin{proof}
	An extended admissible $\rcirc$-dissection of $(\Sigma,M,P)$ dissects the surface into a collection of discs $\{\nabla_1, \ldots, \nabla_t\}$, such that $\nabla_i$ either;
	\begin{enumerate}
		\item has exactly one side contained within $\partial \Sigma$, or;
		\item has no sides contained within $\partial \Sigma$, and a binding arc to some marked point in $P_{\gbullet}$ on the interior of $\nabla_i$.
	\end{enumerate}
	This follows from \cite[Proposition 1.12]{AmiotPlamondonSchroll} and the definition of extended admissible $\rcirc$-dissections.
	
	We construct the collection $\widetilde{M}_{\gbullet}$ by placing a marked point on the side contained within $\partial \Sigma$ for each $\nabla_i$ satisfying (1).
	We construct the collection $\widetilde{P}_{\gbullet}$ by placing a marked point on the interior of each $\nabla_j$ satisfying (2).
	
	Therefore $\Delta$ induces an admissible collection of arcs in $(\Sigma,\widetilde{M},\widetilde{P})$, and it is only left to show that this is an admissible $\rcirc$-dissection.
	By assumption we have $ \lvert \widetilde{M}_{\rcirc} \rvert = \lvert M \rvert$ and $ \lvert \widetilde{P}_{\rcirc} \rvert = \lvert P \rvert$.
	Similarly, by construction, we have $\lvert \widetilde{P}_{\gbullet} \rvert = \lvert P_{\gbullet} \rvert$, as there is exactly one puncture in $P_{\gbullet}$ for all $\nabla_j$ satisfying (2).
	By \Cref{Cor: Arcs in Extended}, we know that $\Delta$ has 
	\[
	\lvert M \rvert + \lvert P \rvert + \lvert P_{\gbullet} \rvert + b + 2g -2
	\]
	arcs, and by the previous discussion we have
	\begin{align*}
		\lvert M \rvert + \lvert P \rvert + \lvert P_{\gbullet} \rvert + b + 2g -2 &= \lvert \widetilde{M}_{\rcirc} \rvert + \lvert \widetilde{P}_{\rcirc} \rvert + \lvert \widetilde{P}_{\gbullet} \rvert +b +2g -2\\
		&= \lvert \widetilde{M}_{\rcirc} \rvert + \lvert \widetilde{P} \rvert + b + 2g -2.
	\end{align*}
	Hence, by \Cref{Prop: arcs in admissible}, $\Delta$ is an admissible $\rcirc$-dissection of $(\Sigma,\widetilde{M},\widetilde{P})$.
	
	Uniqueness up to homeomorphism follows from \cite[Proposition 1.12]{AmiotPlamondonSchroll}, as any disc $\nabla_i$ with a shared boundary component of $\Sigma$ must contain a $\gbullet$-point on the singular boundary component, and any disc without a shared boundary component of $\Sigma$ must contain a $\gbullet$-puncture on the interior.
	The placement of these marked points is determined up to homeomorphism of $(\Sigma, \widetilde{M}, \widetilde{P})$.
\end{proof}

We will say that the admissible $\rcirc$-dissection $\widetilde{\Delta}$ is \textit{induced} by the extended admissible $\rcirc$-dissection $\Delta$.

\begin{corollary}\label[cor]{Cor: Extended to Gentle}
	The bijection of \Cref{Thm: Gentle to Admissible} restricts to an injection of the set of marked surfaces $(\Sigma,M,P)$ with an extended admissible $\rcirc$-dissection into the set of isomorphism classes of locally gentle algebras.
\end{corollary}

Given an extended admissible $\rcirc$-dissection $\Delta$ of a marked surface $(\Sigma,M,P)$, we let $\widetilde{\Delta}$ be the admissible $\circ$-dissection of $(\Sigma,\widetilde{M},\widetilde{P})$ induced by \Cref{Lem: extended become admissible}.
Then we let the quiver with relations associated to $\Delta$ be the same as the quiver with relations associated to $\widetilde{\Delta}$ (\Cref{Def: Quiver from dissection}), that is,
\[
(\widehat{Q^{\Delta}},\widehat{I^{\Delta}}) = (\widehat{Q^{\widetilde{\Delta}}},\widehat{I^{\widetilde{\Delta}}}).
\]

\begin{lemma}\label[lem]{Lem: arcs between binding}
	Let $\Delta$ be an extended admissible $\rcirc$-dissection of $(\Sigma,M,P)$, and let $\delta,\delta'$ be binding arcs that share an $\rcirc$-endpoint.
	Then there exists a $\rcirc$-arc, $\gamma$, fitting into the following diagram;
	
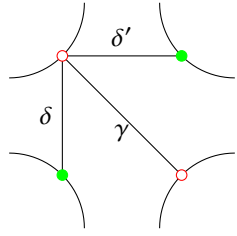
\[
	\begin{tikzpicture}
		\draw[domain=0:90] plot ({cos(\x)-1.5},{sin(\x)-1.5});
		\draw[domain=180:270] plot ({cos(\x)+1.5},{sin(\x)+1.5});
		\draw[domain=90:180] plot ({cos(\x)+1.5},{sin(\x)-1.5});
		\draw[domain=270:360] plot ({cos(\x)-1.5},{sin(\x)+1.5});
		\draw[thin] ({cos(315)-1.5},{sin(315)+1.5}) -- ({cos(45)-1.5}, {sin(45)-1.5}) node[pos=0.5,left] {$\delta$};
		\draw[thin] ({cos(315)-1.5},{sin(315)+1.5}) -- ({cos(225)+1.5}, {sin(225)+1.5}) node[pos=0.5,above] {$\delta'$};
		\draw[thin] ({cos(315)-1.5},{sin(315)+1.5}) -- ({cos(135)+1.5},{sin(135)-1.5}) node[pos=0.5,below] {$\gamma$};
		\draw[red,fill=white] ({cos(315)-1.5},{sin(315)+1.5}) circle (0.07cm);
		\draw[red,fill=white] ({cos(135)+1.5},{sin(135)-1.5}) circle (0.07cm);
		\draw[green,fill=green] ({cos(225)+1.5}, {sin(225)+1.5}) circle (0.07cm);
		\draw[green,fill=green] ({cos(45)-1.5}, {sin(45)-1.5}) circle (0.07cm);
	\end{tikzpicture}
	\]
\end{lemma}

\begin{proof}
	By definition, in an admissible dissection, no two points in $M_{\gbullet}$ may be in the same subsurface cut out by the $\rcirc$-arcs.
	Consequently, any $\gbullet$-arc must intersect a $\rcirc$-arc in the admissible $\rcirc$-dissection.
	Consider the arc;
	
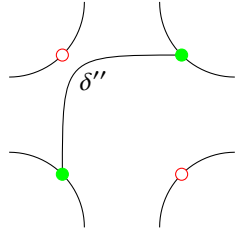
\[
	\begin{tikzpicture}
		\draw[domain=0:90] plot ({cos(\x)-1.5},{sin(\x)-1.5});
		\draw[domain=180:270] plot ({cos(\x)+1.5},{sin(\x)+1.5});
		\draw[domain=90:180] plot ({cos(\x)+1.5},{sin(\x)-1.5});
		\draw[domain=270:360] plot ({cos(\x)-1.5},{sin(\x)+1.5});
		\draw[thin] ({cos(225)+1.5},{sin(225)+1.5}) .. controls ({cos(315)-1.5},{sin(315)+1.5}) .. ({cos(45)-1.5},{sin(45)-1.5}) node[pos=0.6,right] {$\delta''$};
		\draw[red,fill=white] ({cos(315)-1.5},{sin(315)+1.5}) circle (0.08cm);
		\draw[red,fill=white] ({cos(135)+1.5},{sin(135)-1.5}) circle (0.08cm);
		\draw[green,fill=green] ({cos(225)+1.5}, {sin(225)+1.5}) circle (0.08cm);
		\draw[green,fill=green] ({cos(45)-1.5}, {sin(45)-1.5}) circle (0.08cm);
	\end{tikzpicture}
	\]
	Then $\gamma$ is a $\rcirc$-arc that intersects $\delta''$,for all representatives of the homotopy class, and does not intersect $\delta$ and $\delta'$.
	We need that $\gamma$ does not intersect $\delta$ and $\delta'$ as they are all members of the same extended admissible $\rcirc$-dissection.
	
	The $\rcirc$-arc $\gamma$ is not necessarily unique, however any choice must retain the same properties of having an endpoint shared with $\delta$ and $\delta'$, as well as intersecting $\delta''$.
	Hence there must be a $\rcirc$-arc in $\Delta$ that shares the endpoint that is shared by $\delta$ and $\delta'$.
\end{proof}

A natural question to ask would be when given an admissible dissection of a marked surface $(\Sigma,\tilde{M},\tilde{P})$, can we say whether or not this is induced by an extended admissible dissection of some other marked surface $(\Sigma,M,P)$?
The answer to this question is yes, it is possible, and we provide a list of criteria to show exactly when an admissible dissection is induced by an extended admissible dissection.

\begin{lemma}\label[lem]{Lem: Admissible to extended}
	Let $\widetilde{\Delta}$ be an admissible $\rcirc$-dissection of a marked surface $(\Sigma,\widetilde{M},\widetilde{P})$.
	Then there exists an extended admissible $\rcirc$-dissection, $\Delta$, of a marked surface $(\Sigma,M,P)$, such that $\widetilde{\Delta}$ is induced by $\Delta$ if and only if the following conditions hold;
	\begin{enumerate}
		\item $\lvert \widetilde{M}_{\rcirc} \cap S \lvert = 2p$, with $p > 0$, for each boundary component $S \subseteq \partial \Sigma$,
		\item there is exactly one $\rcirc$-arc incident to each alternating point in $\widetilde{M}_{\rcirc} \cap S$, considered under the orientation of $S$,
		\item for each marked point in $\widetilde{P}_{\gbullet}$ contained in a subsurface $\nabla$, then there at least one marked point in $\widetilde{P}_{\rcirc}$ contained in $\partial \nabla$ that is incident to a single $\rcirc$-arc in $\Delta$.
	\end{enumerate}
\end{lemma}

\begin{proof}
We begin by showing that if we have an admissible $\rcirc$-dissection $\widetilde{\Delta}$ satisfying the above conditions, then there exists some extended admissible $\rcirc$-dissection, $\Delta$, of a marked surface $(\Sigma,M,P)$, that induces $\widetilde{\Delta}$.

We relabel the marked points and punctures in $\widetilde{M}_{\rcirc} \cup \widetilde{P}_{\rcirc}$.
By condition $(2)$ above, there exists a subset $M_{\gbullet}$ of $\widetilde{M}_{\rcirc}$ such that each point is incident to one $\rcirc$-arc in $\widetilde{\Delta}$, and there is exactly one point not in $M_{\gbullet}$ between successive points in $M_{\gbullet}$ following the orientation of the relevant boundary component.
We let $M_{\rcirc}$ be the complement of $M_{\gbullet}$ in $\widetilde{M}_{\rcirc}$, and note that this only makes sense in light of condition $(1)$, else $M_{\gbullet}$ and $\widetilde{M}_{\rcirc}$ would coincide when restricted to a boundary component that does not satisfy $(1)$.

Similarly, let $P_{\gbullet}$ be the subset of $\widetilde{P}_{\rcirc}$ made up of points that satisfy condition $(3)$, and let $P_{\rcirc}$ be the complement of $P_{\gbullet}$ in $\widetilde{P}_{\rcirc}$.

It is clear that $(\Sigma,M,P)$ is a marked surface, where $M = M_{\rcirc} + M_{\gbullet}$ and $P = P_{\rcirc} + P_{\gbullet}$.
Now let $\widetilde{\Delta}'$ be the set of $\rcirc$-arcs on $(\Sigma,M,P)$ that correspond to the $\rcirc$-arcs in $\widetilde{\Delta}$ that do not have an endpoint at any of the marked points or punctures in the relabelled subsets $M_{\gbullet}$ or $P_{\gbullet}$.
Then $\widetilde{\Delta}'$ is an admissible collection of $\rcirc$-arcs of $(\Sigma,M,P)$, as every subsurface cut from $\widetilde{\Delta}'$ either contains exactly one puncture in $P_{\gbullet}$, as we replaced the puncture from $\widetilde{P}_{\gbullet}$ in the admissible $\rcirc$-dissection $\widetilde{\Delta}$, or contains exactly one marked point in $M_{\gbullet}$ on its boundary, by the construction of $M_{\gbullet}$.

The fact that there exists a single point in $M_{\gbullet}$ follows from the fact the corresponding point $p \in \widetilde{M}_{\rcirc}$ is incident to a single $\rcirc$-arc, and so this $\rcirc$-arc bounds two subsurfaces, both sharing a boundary component with $\partial \Sigma$, and both of those boundary components containing the point $p$.
By removing this $\rcirc$-arc, we induce a subsurface that corresponds to the union of the two subsurfaces above, and this subsurface must have a single boundary component, between $p^-$ and $p^+$, the predecessor and successor of $p$ in $\widetilde{M}_{\rcirc}$ with respect to the orientation of the boundary component of $\partial \Sigma$.

However, we have
\begin{align*}
	\lvert \widetilde{\Delta}' \rvert &= \lvert \widetilde{\Delta} \rvert - \lvert M_{\gbullet} \rvert - \lvert P_{\gbullet} \rvert\\
	&= \lvert \widetilde{M}_{\rcirc} \rvert + \lvert \widetilde{P} \rvert + b + 2g -2 - \lvert M_{\gbullet} \rvert - \lvert P_{\gbullet} \rvert\\
	&= \lvert M_{\rcirc} \rvert + \lvert M_{\gbullet} \rvert + \lvert P_{\rcirc} \rvert + 2\lvert P_{\gbullet} \rvert + b + 2g -2 - \lvert M_{\gbullet} \rvert - \lvert P_{\gbullet} \rvert\\
	&= \lvert M_{\rcirc} \rvert + \lvert P \rvert + b + 2g -2,
\end{align*}
where the third line follows from $\lvert \widetilde{P}_{\gbullet} \rvert = \lvert P_{\gbullet} \rvert$ by condition $(3)$.
Hence, by \Cref{Prop: arcs in admissible}, $\widetilde{\Delta}'$ is an admissible $\rcirc$-dissection of $(\Sigma,M,P)$.

Therefore, $\Delta$ is an extended admissible $\rcirc$-dissection as it is the data of an admissible $\rcirc$-dissection with a binding arc for each marked point in $M_{\gbullet}$ and each puncture in $P_{\gbullet}$, guaranteed by conditions (2) and (3) respectively.

To prove the other direction, let $\Delta$ be an extended admissible $\rcirc$-dissection of $(\Sigma,M,P)$.
Then by construction of the marked surface $(\Sigma,\widetilde{M},\widetilde{P})$ in \Cref{Lem: extended become admissible}, we see that $M \mapsto \widetilde{M}_{\rcirc}$, and so condition $(1)$ follows.

Given there is exactly one binding arc incident to any marked point in $M_{\gbullet}$, and $M_{\gbullet}$ and $M_{\rcirc}$ alternate on any boundary component of $\Sigma$, we see that condition $(2)$ holds.

Finally, by \cite[Propositon 1.12]{AmiotPlamondonSchroll}, we see that any puncture in $P_{\gbullet}$ is isolated in a subsurface cut out by the $\rcirc$-arcs in the admissible $\rcirc$-dissection underlying $\Delta$.
Hence there is at least one puncture in $\widetilde{P}_{\rcirc}$ for any subsurface cut out from $\widetilde{\Delta}$ that only intersects the boundary of $\Sigma$ at finitely many points, such that the puncture is incident to a single $\rcirc$-arc, and so condition $(3)$ holds.
\end{proof}

\subsection{Piano Algebras}

Let $\Delta$ be an extended admissible $\rcirc$-dissection for some marked surface $(\Sigma,M,P)$.
Then we call the data of the associated quiver with relations, $(\widehat{Q^{\Delta}}, \widehat{I^{\Delta}},B^{\Delta})$, a \textit{keyboard quiver}, where $B^{\Delta}$ is the set of vertices in $\widehat{Q^{\Delta}}$ corresponding to binding arcs.
A vertex in $(\widehat{Q^{\Delta}}, \widehat{I^{\Delta}},B^{\Delta})$ corresponding to a binding arc is called a \textit{sharp vertex}, and a vertex corresponding to a $\rcirc$-arc is a \textit{standard vertex}.
We shall denote a sharp vertex by a black point, $\bullet$, and a standard vertex by a white point, $\circ$.

We shall call the algebra $k\widehat{Q^{\Delta}} / \widehat{I^{\Delta}}$ a \textit{keyboard algebra}.
We shall say that two extended admissible $\rcirc$-dissections are \textit{equivalent} if their respective keyboard algebras are isomorphic.

In general, given a locally gentle algebra, it proves remarkably difficult to provide sufficient conditions to show that it is a keyboard algebra without passing to the associated admissible $\rcirc$-dissection and applying \Cref{Lem: Admissible to extended}.

\begin{example}\label[ex]{Ex: A keyboard}
	Let $\Sigma$ be a disc, $\lvert M_{\rcirc} \rvert = 5$ and $\lvert P \rvert =0$.
	Then we have the following extended admissible $\rcirc$-dissection, which we label $\Delta$.
	\[
	\begin{tikzpicture}
		\draw (0,0) circle (3cm);
		\draw ({3*sin(0)},{3*cos(0)}) -- ({3*sin(216)},{3*cos(216)}) node[pos=0.5] (1) {$\empty$};
		\draw ({3*sin(0)},{3*cos(0)}) -- ({3*sin(324)},{3*cos(324)}) node[pos=0.5] (2) {$\empty$};
		\draw ({3*sin(0)},{3*cos(0)}) -- ({3*sin(72)},{3*cos(72)}) node[pos=0.5] (3) {$\empty$};
		\draw ({3*sin(252)},{3*cos(252)}) -- ({3*sin(216)},{3*cos(216)}) node[pos=0.5] (4) {$\empty$};
		\draw ({3*sin(288)},{3*cos(288)}) -- ({3*sin(216)},{3*cos(216)}) node[pos=0.5] (5) {$\empty$};
		\draw ({3*sin(180)},{3*cos(180)}) -- ({3*sin(216)},{3*cos(216)}) node[pos=0.5] (6) {$\empty$};
		\draw ({3*sin(36)},{3*cos(36)}) -- ({3*sin(72)},{3*cos(72)}) node[pos=0.5] (7) {$\empty$};
		\draw ({3*sin(144)},{3*cos(144)}) -- ({3*sin(72)},{3*cos(72)}) node[pos=0.5] (8) {$\empty$};
		\draw ({3*sin(144)},{3*cos(144)}) -- ({3*sin(108)},{3*cos(108)}) node[pos=0.5] (9) {$\empty$};
		\draw[red,fill=white] ({3*sin(0)},{3*cos(0)}) circle (0.1cm);
		\draw[red,fill=white] ({3*sin(72)},{3*cos(72)}) circle (0.1cm);
		\draw[red,fill=white] ({3*sin(144)},{3*cos(144)}) circle (0.1cm);
		\draw[red,fill=white] ({3*sin(216)},{3*cos(216)}) circle (0.1cm);
		\draw[red,fill=white] ({3*sin(288)},{3*cos(288)}) circle (0.1cm);
		\draw[fill=green,green] ({3*sin(36)},{3*cos(36)}) circle (0.1cm);
		\draw[fill=green,green] ({3*sin(108)},{3*cos(108)}) circle (0.1cm);
		\draw[fill=green,green] ({3*sin(180)},{3*cos(180)}) circle (0.1cm);
		\draw[fill=green,green] ({3*sin(252)},{3*cos(252)}) circle (0.1cm);
		\draw[fill=green,green] ({3*sin(324)},{3*cos(324)}) circle (0.1cm);
		\draw[red,dashed,->] (2) -- (1) node[pos=0.75] (A) {$\empty$};
		\draw[red,dashed,->] (6) -- (1) node[pos=0.83] (B) {$\empty$};
		\draw[red,dashed,->] (1) -- (5) node[pos=0.25] (C) {$\empty$};
		\draw[red,dashed,->] (1) -- (3) node[pos=0.25] (D) {$\empty$} node[pos=0.75] (E) {$\empty$};
		\draw[red,dashed,->] (5) -- (4);
		\draw[red,dashed,->] (3) -- (8) node[pos=0.25] (F) {$\empty$};
		\draw[red,dashed,->] (9) -- (8);
		\draw[red,dashed,->] (7) -- (3);
		\draw[red,dotted] (A) to [out=180,in=135] (C);
		\draw[red,dotted] (B) to [out=15,in=280] (D);
		\draw[red,dotted] (E) to [out=290,in=220] (F);
	\end{tikzpicture}
	\]
	Accordingly, we have the associated keyboard quiver $(\widehat{Q^{\Delta}},\widehat{I^{\Delta}},B^{\Delta})$, where the relations of $\widehat{I^{\Delta}}$ are denoted by the dotted paths between arrows.
	\[
\begin{tikzcd}
	\bullet \arrow[drr, ""{name=UL,pos=0.75}] &&&& \circ \arrow[rr] && \bullet &&\\
	&& \circ \arrow[urr,""{name=UR,pos=0.25}] \arrow[drr,""{name=LC,pos=0.25},""{name=LCR,pos=0.75}] &&&& \circ && \bullet \arrow[ll]\\
	\bullet \arrow[urr,""{name=LL,pos=0.75}] &&&& \circ \arrow[urr,""{name=LR,pos=0.25}] && \bullet  \arrow[ll] &&
	\arrow[from=UL, to=UR,dotted, bend left=30,-]
	\arrow[from=LL,to=LC,dotted, bend right=30,-, shorten <= 5pt, shorten >= 5pt]
	\arrow[from=LCR,to=LR,dotted, bend left=30,-]
\end{tikzcd}
\]
\end{example}

\begin{question}
Let $(\widehat{Q^{\Delta}},\widehat{I^{\Delta}},B^{\Delta})$ be a keyboard quiver associated to some extended admissible $\rcirc$-dissection, $\Delta$, of a marked surface. Does there exist another extended admissible $\rcirc$-dissection, $\Delta'$, such that $\widehat{Q^{\Delta}} = \widehat{Q^{\Delta'}}$, and $\widehat{I^{\Delta}} = \widehat{I^{\Delta'}}$, but $B^{\Delta} \neq B^{\Delta'}$?
Equivalently, for a keyboard quiver $(\widehat{Q},\widehat{I},B)$, is $B$ uniquely determined? If not, can we provide any criteria for when $B$ must be unique?
\end{question}

Let $\Delta$ be an extended admissible $\rcirc$-dissection of a marked surface $(\Sigma,M,P)$, and let $(\widehat{Q^{\Delta}}, \widehat{I^{\Delta}}, B^{\Delta})$ be the associated keyboard quiver.
Then we define a triple $(Q^{\Delta},I^{\Delta},B^{\Delta})$ as follows.
We let $Q_0^{\Delta} = \widehat{Q_0^{\Delta}}$, and $\widehat{Q_1^{\Delta}} \subset Q_1^{\Delta}$, where we add in an arrow $\alpha_a$ for all $a \in Q_0^{\Delta}$, and an arrow $\beta_b$ for all $b \not\in B$.
We impose a grading on $(Q^{\Delta},I^{\Delta},B^{\Delta})$ by placing  $\alpha_a$ in degree $-1$, and $\beta_b$ in degree 1. 

The set of relations $I^{\Delta}$ is generated by the relations,
\[
I^{\Delta} := \langle \alpha_{a} \beta_{a}-\iota_{a},\, \beta_{a} \alpha_{a}-\iota_{a},\, \alpha_a \delta_{ab} - \delta_{ab} \alpha_b,\, \beta_{a'} \delta_{a'a} \cdots \delta_{bb'} - \delta_{a'a} \cdots \delta_{bb'} \beta_{b'}, \widehat{I^{\Delta}} \rangle,
\]
for all such paths that exist in $Q^{\Delta}$.
Here $\iota_a$ denotes the trivial path at vertex $a$.
We may drop the notation of the vertices for $\alpha,\beta$ and $\delta$, where it is implicit that each $\alpha$ (resp.\ $\beta$) in $\alpha \delta - \delta \alpha$ (resp.\ $\beta \delta - \delta \beta$) is based at a different vertex determined by $\delta$.
We call the triple $(Q^{\Delta},I^{\Delta},B^{\Delta})$ a \textit{piano quiver}.

\begin{definition}
	Let $(\Sigma,M,P)$ be a marked surface, and let $\Delta$ be an extended admissible $\rcirc$-dissection of the marked surface $(\Sigma,M,P)$.
	Then we call the differentially graded $k$-algebra $\Lambda^{\Delta} = (kQ^{\Delta}/I^{\Delta},d=0)$ the \textit{piano algebra associated to} $\Delta$.
\end{definition}

We may analogously define a piano algebra from an extended admissible $\gbullet$-dissection of $(\Sigma,M,P)$, however this adds no new examples of piano algebras as this only differs up to the choice of colouring of the marked points and punctures of the marked surface.

\subsection{Classical Generators and Admissible Dissections of a Disc}

Throughout this section, we use the notation $(D^2,M_n,\emptyset)$ to mean the marked, unpunctured disc with $\lvert M_{\rcirc} \rvert = \lvert M_{\gbullet} \rvert = n$.
We also consider the marked disc $(D^2,\mathscr{M})$ from \Cref{Sec: Those cats} to be a marked surface $(D^2,M,\emptyset)$, where accumulation points correspond to points in $M_{\rcirc}$, and their cyclic ordering is preserved.

\begin{definition}
	Let $H = \bigoplus_{i=1}^l H_i$ be an object in $\ocC$ such that each $H_i$ is indecomposable and corresponds to a double limit arc.
	Then we say $H$ is a \textit{limit pre-generator} if $H$ is homologically connected, the collection of endpoints of $H$ is $L(\mathscr{M})$, and no two arcs cross.
	We also ask that $H$ is minimal with respect to satisfying these conditions, that is, no proper direct summand of $H$ satisfies these conditions.
	
	We say that two limit pre-generators $H,H'$ are \textit{equivalent} if 
	\[
	\mathrm{add}(H[i] \mid i \in \mathbb{Z}) \xrightarrow{\sim} \mathrm{add}(H'[i]\mid i \in \mathbb{Z})
	\]
	are equivalent as additive categories.
\end{definition}

The naming convention may seem unusual, given that limit pre-generators are not in fact classical generators of $\ocC$, as they do not have a complete orbit in $\overline{\mathscr{M}}$.
However, we will later justify the naming by seeing that every \textit{limit generator} necessarily has a limit pre-generator as a direct summand, and every limit pre-generator is a direct summand of a limit generator.

We also note the implication that any arc in a limit pre-generator $H \in \ocC$ must share an endpoint with another arc in $H$.
This follows from $H$ being homologically connected and no two arcs in $H$ crossing.

\begin{lemma}\label[lem]{Lem: Bijection of admissible dissections}
	There is a bijection between the homeomorphism class of admissible dissections of $(D^2,M_n,\emptyset)$ and equivalence class of limit pre-generators of $\ocC$.
\end{lemma}

\begin{proof}
	We use the assignment $\varepsilon$ taking an accumulation point to a point in $M_{\rcirc}$, and the set of marked point in $\mathscr{M}$ between two successive accumulation points $\mathfrak{a}$ and $\mathfrak{a}^+$ to a point in $M_{\gbullet}$, such that $\varepsilon$ respects the orientation of the disc.
	For an object in $X \in \ocC$, we write $\varepsilon(X)$ to denote the collection of $\rcirc$-arcs induced by the assignment $\varepsilon$.
	
	Then a limit pre-generator $H$ defines a collection of non-intersecting and pairwise non-isomorphic arcs, and so $\varepsilon(H)$ is a collection of non-intersecting and pairwise  $\rcirc$-arcs.
	By \Cref{Prop: arcs in admissible}, it is enough to show that $H$ has $n-1$ indecomposable direct summands to show that $\varepsilon(H)$ is an admissible $\rcirc$-dissection.
	However, as $H$ must be homologically connected, have $n$ endpoints, and is minimal with respect to these properties, then $\varepsilon(H)$ corresponds to a tree with $n$ vertices by considering the non-orientated quiver induced by $\varepsilon(H)$.
	Then basic results in graph theory tell us that there must be exactly $n-1$ $\rcirc$-arcs in the collection $\varepsilon(H)$, hence $\varepsilon(H)$ is an admissible $\rcirc$-dissection for any limit pre-generator $H$.
	
	Suppose that $H,H'$ are not equivalent limit pre-generators.
	As $H$ and $H'$ have the same number of indecomposable direct summands, all of which are double limit arcs and so have isomorphic endomorphism rings, then for any assignment $\eta: \mathrm{add}(H[i]) \rightarrow \mathrm{add}(H'[i])$, there exists at least two indecomposable direct summands $H_1,H_2$ of $H$ such that
	\[
	\ocC(H_1,H_2) \not\cong \ocC (\eta(H_1),\eta(H_2)).
	\]
	We know $H_1,H_2$ are indecomposable, and $\mathrm{Hom}$-spaces of indecomposable objects in $\ocC$ have dimension either 0 or 1 over $k$, and this corresponds to double limit arcs sharing an accumulation as an endpoint in a particular orientation in our situation.
	Therefore there exists $H_1,H_2$ indecomposable direct summands of $H$ such that $\varepsilon(H_1)$ and $\varepsilon(H_2)$ share an endpoint, with $\varepsilon(H_2)$ a rotation of $\varepsilon(H_1)$ in the positive orientation if and only if this is not true for $\varepsilon(\eta(H_1))$ and $\varepsilon(\eta(H_2))$, for any assignment $\eta$.
	Hence $\varepsilon(H)$ and $\varepsilon(H')$ are non-equivalent admissible dissections.
	
	Finally, given an admissible dissection $\Delta$, it is straightforward to see that $\varepsilon^{-1}(\Delta)$ is a collection of $n-1$ non-intersecting and pairwise different double limit arcs, and so is a limit pre-generator of $\ocC$.
	Thus we have a bijection between the equivalence class of limit pre-generators of $\ocC$ and the homeomorphism class of admissible dissections of $(D^2,M_n,\emptyset)$.
\end{proof}

\begin{example}\label[ex]{Ex: a gentle quiver}
	Here we have an example of a gentle quiver with $9$ vertices and without loops or cycles, and a collection of double limit arcs corresponding to the gentle quiver via the admissible $\rcirc$-dissection in \Cref{Ex: A keyboard}.
	\[
	\begin{tikzcd}
		\circ \arrow[drr, ""{name=UL,pos=0.75}] &&&& \circ \arrow[rr] && \circ &&\\
		&& \circ \arrow[urr,""{name=UR,pos=0.25}] \arrow[drr,""{name=LC,pos=0.25},""{name=LCR,pos=0.75}] &&&& \circ && \circ \arrow[ll]\\
		\circ \arrow[urr,""{name=LL,pos=0.75}] &&&& \circ \arrow[urr,""{name=LR,pos=0.25}] && \circ \arrow[ll] &&
		\arrow[from=UL, to=UR,dotted, bend left=30,-]
		\arrow[from=LL,to=LC,dotted, bend right=30,-, shorten <= 5pt, shorten >= 5pt]
		\arrow[from=LCR,to=LR,dotted, bend left=30,-]
	\end{tikzcd}
	\]
	\[
	\begin{tikzpicture}
		\draw (0,0) circle (3cm);
		\draw ({3*sin(0)},{3*cos(0)}) -- ({3*sin(216)},{3*cos(216)}) node[pos=0.5] (1) {$\empty$};
		\draw ({3*sin(0)},{3*cos(0)}) -- ({3*sin(324)},{3*cos(324)}) node[pos=0.5] (2) {$\empty$};
		\draw ({3*sin(0)},{3*cos(0)}) -- ({3*sin(72)},{3*cos(72)}) node[pos=0.5] (3) {$\empty$};
		\draw ({3*sin(252)},{3*cos(252)}) -- ({3*sin(216)},{3*cos(216)}) node[pos=0.5] (4) {$\empty$};
		\draw ({3*sin(288)},{3*cos(288)}) -- ({3*sin(216)},{3*cos(216)}) node[pos=0.5] (5) {$\empty$};
		\draw ({3*sin(180)},{3*cos(180)}) -- ({3*sin(216)},{3*cos(216)}) node[pos=0.5] (6) {$\empty$};
		\draw ({3*sin(36)},{3*cos(36)}) -- ({3*sin(72)},{3*cos(72)}) node[pos=0.5] (7) {$\empty$};
		\draw ({3*sin(144)},{3*cos(144)}) -- ({3*sin(72)},{3*cos(72)}) node[pos=0.5] (8) {$\empty$};
		\draw ({3*sin(144)},{3*cos(144)}) -- ({3*sin(108)},{3*cos(108)}) node[pos=0.5] (9) {$\empty$};
		\draw[fill=white] ({3*sin(0)},{3*cos(0)}) circle (0.1cm);
		\draw[fill=white] ({3*sin(36)},{3*cos(36)}) circle (0.1cm);
		\draw[fill=white] ({3*sin(72)},{3*cos(72)}) circle (0.1cm);
		\draw[fill=white] ({3*sin(108)},{3*cos(108)}) circle (0.1cm);
		\draw[fill=white] ({3*sin(144)},{3*cos(144)}) circle (0.1cm);
		\draw[fill=white] ({3*sin(180)},{3*cos(180)}) circle (0.1cm);
		\draw[fill=white] ({3*sin(216)},{3*cos(216)}) circle (0.1cm);
		\draw[fill=white] ({3*sin(252)},{3*cos(252)}) circle (0.1cm);
		\draw[fill=white] ({3*sin(288)},{3*cos(288)}) circle (0.1cm);
		\draw[fill=white] ({3*sin(324)},{3*cos(324)}) circle (0.1cm);
		\draw[red,dashed,->] (2) -- (1) node[pos=0.75] (A) {$\empty$};
		\draw[red,dashed,->] (6) -- (1) node[pos=0.83] (B) {$\empty$};
		\draw[red,dashed,->] (1) -- (5) node[pos=0.25] (C) {$\empty$};
		\draw[red,dashed,->] (1) -- (3) node[pos=0.25] (D) {$\empty$} node[pos=0.75] (E) {$\empty$};
		\draw[red,dashed,->] (5) -- (4);
		\draw[red,dashed,->] (3) -- (8) node[pos=0.25] (F) {$\empty$};
		\draw[red,dashed,->] (9) -- (8);
		\draw[red,dashed,->] (7) -- (3);
		\draw[red,dotted] (A) to [out=180,in=135] (C);
		\draw[red,dotted] (B) to [out=15,in=280] (D);
		\draw[red,dotted] (E) to [out=290,in=220] (F);
	\end{tikzpicture}
	\]
\end{example}

\begin{definition}
	Let $G \cong \bigoplus_{a=1}^m G_a$ be a minimal generator of $\ocC$.
	We call $G$ a \textit{limit generator}, if $G_a$ is a (double) limit arc, and $G_a$ and $G_b[i]$ do not cross each other for any $i \in \mathbb{Z}$, for all $1 \leq a,b \leq m$.
\end{definition}

\begin{remark}\label[rem]{Rem: One arc per section}
A straightforward observation from the definition of a limit generator tells us that there is exactly one limit arc with an endpoint in the open interval $(a_i,a_{i+1}) \subset \overline{\mathscr{M}}$ for all $1 \leq i \leq n$.
This is as if two arcs had an endpoint in such an open interval, then some suspension of the arcs would cross each other, or be isomorphic, and so would not be a limit generator.
There must also be at least one arc with an endpoint in the open interval, as $G$ must have a complete orbit in $\overline{\mathscr{M}}$ to be a classical generator by \Cref{Thm:Gens}.
\end{remark}

\begin{lemma}
	Let $G \in \ocC$ be a limit generator.
	Then $G \cong \prescript{p}{\empty}{G} \oplus \prescript{l}{\empty}{G}$, where $\prescript{p}{\empty}{G}$ is a limit pre-generator with $n-1$ indecomposable direct summands, and $\prescript{l}{\empty}{G}$ is a set of $n$ non-isomorphic, non-intersecting limit arcs.
\end{lemma}

\begin{proof}
	By definition, all indecomposable direct summands of $G$ are either limit arcs, or double limit arcs, such that no two arcs intersect up to suspension, but may cross.
	Therefore, for each interval $(\mathfrak{a}, \mathfrak{a}^+) \subset \overline{\mathscr{M}}$, there must be exactly one limit arc $G_a$ with an endpoint in $(\mathfrak{a}, \mathfrak{a}^+)$, see \Cref{Rem: One arc per section}.
	Label the collection of limit arcs in $G$ by $\prescript{l}{\empty}{G}$.
	If two limit arcs in $\prescript{l}{\empty}{G}$ intersect up to suspension, then $G$ is not a limit generator by definition, and if they are isomorphic up to suspension, then $G$ is no longer minimal as a proper direct summand of $G$ is a generator, and so $G$ would not a limit generator by definition.
	Hence there is a sub-collection $\prescript{l}{\empty}{G}$ of $n$ non-intersecting limit arcs in $G$, one for each subset $(\mathfrak{a}, \mathfrak{a}^+) \subset \overline{\mathscr{M}}$.
	
	All other indecomposable direct summands of $G$ must then be double limit arcs, so we need to show that there are $n-1$ of these direct summands, that the object $\prescript{p}{\empty}{G}$ is homologically connected and has a collection of endpoints equal to $L(\mathscr{M})$.
	Let $\mathfrak{a}_1,\mathfrak{a}_m$ be distinct accumulation points, and let $G_a,G_b$ be limit arcs with an endpoint at $\mathfrak{a}_1$ and $\mathfrak{a}_m$ respectively.
	Then $G_a,G_b$ do not cross, and $G$ may only be homologically connected if there exists a set of double limit arcs $G_{a_1}, \ldots, G_{a_p}$ such that $G_{a_i} = \{\mathfrak{a}_{r_i},\mathfrak{a}_{r_{i+1}}\}$.
	This is because no two distinct arcs in $G$ intersect, and limit arcs will only cross other arcs with the same accumulation point.
	Hence, between any two accumulation points, there is a path that transverses a set of double limit arcs.
	As $G$ is minimal, there may only be a single path of this kind by \cite[Lemma 5.4]{Generators}, and so we can consider the accumulation points as vertices, and the double limit arcs as edges on a connected graph.
	As there is exactly one path between any two vertices, then the graph is in fact a tree with $n$ vertices, and so has $n-1$ edges, meaning there must be $n-1$ double limit arcs as indecomposable direct summands of $G$.
	
	The above arguments also show that the collection of endpoints of $\prescript{p}{\empty}{G}$ is $L(\mathscr{M})$, and that $\prescript{p}{\empty}{G}$ must be homologically connected.
	Hence, $\prescript{p}{\empty}{G}$ is a limit pre-generator.
\end{proof}

\begin{lemma}
	Let $\prescript{p}{\empty}{G} \in \ocC$ be a limit pre-generator.
	Then there exists a limit generator $G \in \ocC$ such that $G \cong \prescript{p}{\empty}{G} \oplus \prescript{l}{}{G}$.
\end{lemma}

\begin{proof}
	Let $\prescript{l}{\empty}{G}$ be collection of limit arcs such that each indecomposable direct summand $\prescript{l}{\empty}{G}_a$ of $\prescript{l}{\empty}{G}$ has endpoints $\{\mathfrak{a}_a,x_a\}$, where $x_a \in (\mathfrak{a}_a,\mathfrak{a}_a^+)$.
	Then it is a simple observation to see that $\prescript{p}{\empty}{G} \oplus \prescript{l}{}{G}$ is a collection of non-intersecting limit and double limit arcs, such that $\prescript{p}{\empty}{G} \oplus \prescript{l}{}{G}$ is homologically connected and has a complete orbit in $\overline{\mathscr{M}}$, and hence is a generator of $\ocC$ by \Cref{Thm:Gens}.
\end{proof}

The above lemma is not an exhaustive list of limit generators containing the limit pre-generator $\prescript{p}{\empty}{G}$.
It is only intended to show the existence of a limit generator for any limit pre-generator.

\begin{proposition}\label[prop]{Prop: Bijection of ext admissible dissections}
	There is a bijection between the set of equivalence classes of limit generators in $\ocC$ and the homeomorphism class of extended admissible $\rcirc$-dissections of $(D^2,M_n,\emptyset)$.
\end{proposition}

\begin{proof}
	Let $\varepsilon$ be the assignment from $\overline{\mathscr{M}}$ to $M_n$ as in the proof of \Cref{Lem: Bijection of admissible dissections}.
	We again denote $\varepsilon(G)$ the collection of $\rcirc$-arcs and binding arcs induced by a limit generator $G \in \ocC$.
	The description of the assignment $\varepsilon$ implies that limit arcs in $G$ are taken to binding arcs, and double limit arcs are taken to $\rcirc$-arcs.
	
	By \Cref{Lem: Bijection of admissible dissections}, we know that $\varepsilon(\empty^p G)$ is an admissible dissection of $(D^2,M_n, \emptyset)$, so we need to show that the limit arcs of $G$ map to a collection of binding arcs such that $\varepsilon(G)$ is an extended admissible dissection.
	To do this, observe that no limit arc of $G$ intersects another arc in $G$, and so $\varepsilon(G)$ is a collection of non-intersecting $\rcirc$-arcs and binding arcs.
	There are also $n$ limit arcs of $G$, as there is only one for each subset $(\mathfrak{a}_i,\mathfrak{a}_{i+1}) \subset \overline{\mathscr{M}}$ by \Cref{Rem: One arc per section}, and so there are $n$ non-intersecting and pairwise different binding arcs in $\varepsilon(G)$.
	Hence, $\varepsilon(G)$ is a collection of $2n-1$ non-intersecting and pairwise different $\rcirc$-arcs and binding arcs, and so is an extended admissible dissection by \Cref{Cor: Arcs in Extended}.
	
	Analogous arguments to the proof of Lemma \ref{Lem: Bijection of admissible dissections} show that for two non-equivalent limit generators $G,G'$, then $\varepsilon(G)$ and $\varepsilon(G')$ are not homeomorphic, and that each extended admissible dissection of $(D^2,M_n,\emptyset)$ comes from a limit generator of $\ocC$, by considering when $\varepsilon^{-1}$ takes a point in $M_{\gbullet}$ to a single point $x \in (\mathfrak{a}_i,\mathfrak{a}_{i+1})$.
\end{proof}

\begin{corollary}\label[cor]{Cor: Indecomposables in limit gens}
	A limit generator $G \in \ocC$ has $2n-1$ indecomposable direct summands.
\end{corollary}

\begin{proof}
	This is a direct consequence of the bijection of \Cref{Prop: Bijection of ext admissible dissections}, and the number of arcs in an extended admissible $\rcirc$-dissection of $(D^2,M_n,\emptyset)$ given by \Cref{Cor: Arcs in Extended}.
\end{proof}

Let $G \in \ocC$ be a limit generator, and let $\Delta$ be the associated extended admissible $\rcirc$-dissection of $(D^2,M_n,\emptyset)$ under the bijection of \Cref{Prop: Bijection of ext admissible dissections}.
Then we will also use the notation $(Q^G,I^G,B^G)$, resp.\ $(\widehat{Q^G},\widehat{I^G},B^G)$, for the piano quiver $(Q^{\Delta},I^{\Delta},B^{\Delta})$, resp.\ the keyboard quiver $(\widehat{Q^{\Delta}},\widehat{I^{\Delta}},B^{\Delta})$.

\begin{example}\label[ex]{Ex: a gentle quiver}
	We retain our keyboard quiver used in \Cref{Ex: A keyboard}.
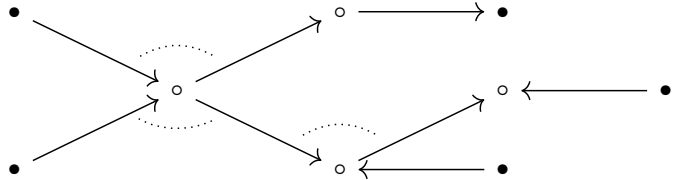
	\[
	\begin{tikzcd}
		\bullet \arrow[drr, ""{name=UL,pos=0.75}] &&&& \circ \arrow[rr] && \bullet &&\\
		&& \circ \arrow[urr,""{name=UR,pos=0.25}] \arrow[drr,""{name=LC,pos=0.25},""{name=LCR,pos=0.75}] &&&& \circ && \bullet \arrow[ll]\\
		\bullet \arrow[urr,""{name=LL,pos=0.75}] &&&& \circ \arrow[urr,""{name=LR,pos=0.25}] && \bullet \arrow[ll] &&
		\arrow[from=UL, to=UR,dotted, bend left=30,-]
		\arrow[from=LL,to=LC,dotted, bend right=30,-, shorten <= 5pt, shorten >= 5pt]
		\arrow[from=LCR,to=LR,dotted, bend left=30,-]
	\end{tikzcd}
	\]
	 The corresponding limit generator of $\ocC$ under \Cref{Prop: Bijection of ext admissible dissections} is as follows.
	
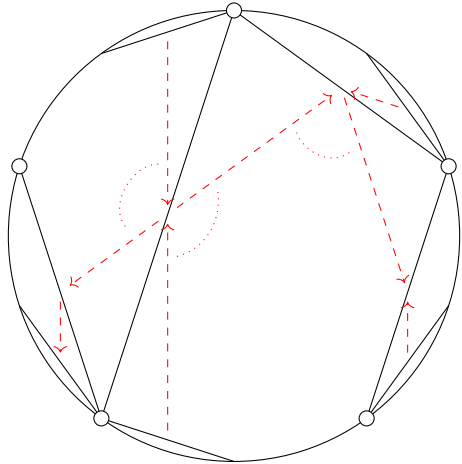
\[
	\begin{tikzpicture}
		\draw (0,0) circle (3cm);
		\draw ({3*sin(0)},{3*cos(0)}) -- ({3*sin(216)},{3*cos(216)}) node[pos=0.5] (1) {$\empty$};
		\draw ({3*sin(0)},{3*cos(0)}) -- ({3*sin(324)},{3*cos(324)}) node[pos=0.5] (2) {$\empty$};
		\draw ({3*sin(0)},{3*cos(0)}) -- ({3*sin(72)},{3*cos(72)}) node[pos=0.5] (3) {$\empty$};
		\draw ({3*sin(252)},{3*cos(252)}) -- ({3*sin(216)},{3*cos(216)}) node[pos=0.5] (4) {$\empty$};
		\draw ({3*sin(288)},{3*cos(288)}) -- ({3*sin(216)},{3*cos(216)}) node[pos=0.5] (5) {$\empty$};
		\draw ({3*sin(180)},{3*cos(180)}) -- ({3*sin(216)},{3*cos(216)}) node[pos=0.5] (6) {$\empty$};
		\draw ({3*sin(36)},{3*cos(36)}) -- ({3*sin(72)},{3*cos(72)}) node[pos=0.5] (7) {$\empty$};
		\draw ({3*sin(144)},{3*cos(144)}) -- ({3*sin(72)},{3*cos(72)}) node[pos=0.5] (8) {$\empty$};
		\draw ({3*sin(144)},{3*cos(144)}) -- ({3*sin(108)},{3*cos(108)}) node[pos=0.5] (9) {$\empty$};
		\draw[fill=white] ({3*sin(0)},{3*cos(0)}) circle (0.1cm);
		\draw[fill=white] ({3*sin(72)},{3*cos(72)}) circle (0.1cm);
		\draw[fill=white] ({3*sin(144)},{3*cos(144)}) circle (0.1cm);
		\draw[fill=white] ({3*sin(216)},{3*cos(216)}) circle (0.1cm);
		\draw[fill=white] ({3*sin(288)},{3*cos(288)}) circle (0.1cm);
		\draw[red,dashed,->] (2) -- (1) node[pos=0.75] (A) {$\empty$};
		\draw[red,dashed,->] (6) -- (1) node[pos=0.83] (B) {$\empty$};
		\draw[red,dashed,->] (1) -- (5) node[pos=0.25] (C) {$\empty$};
		\draw[red,dashed,->] (1) -- (3) node[pos=0.25] (D) {$\empty$} node[pos=0.75] (E) {$\empty$};
		\draw[red,dashed,->] (5) -- (4);
		\draw[red,dashed,->] (3) -- (8) node[pos=0.25] (F) {$\empty$};
		\draw[red,dashed,->] (9) -- (8);
		\draw[red,dashed,->] (7) -- (3);
		\draw[red,dotted] (A) to [out=180,in=135] (C);
		\draw[red,dotted] (B) to [out=15,in=280] (D);
		\draw[red,dotted] (E) to [out=290,in=220] (F);
	\end{tikzpicture}
	\]
\end{example}

\section{Graded Endomorphism Ring of Generators in \texorpdfstring{$\ocC$}{Cn}}\label{Sec:Gens}

In this section we look at the graded endomorphism ring of some generators in $\ocC$.
We do not look at all generators, but cover all generators which satisfy some combinatoric criteria, which includes all limit generators.

\subsection{Graded Endomorphism Rings}

We begin by looking at the graded endomorphism ring of indecomposable objects, with the exception of short arcs, before considering objects without short arcs as direct summands.
However, first let us recall the definition of a graded endomorphism ring.

\begin{definition}
Let $M$ be an object of a triangulated category $\cT$, then the \textit{graded endomorphism ring of} $M$, denoted $\Endo[\cT]{\ast}{M}$, is the graded ring with $i^{th}$ degree
\[
\Endo[\cT]{i}{M} = \Ext[\cT]{i}{M}{M} \cong \cT(M)(M[i]).
\]

Let $f \in \Endo[\cT]{i}{M}$ and $g \in \Endo[\cT]{j}{M}$.
Then we define multiplication as
\[
gf := g [i] \circ f \in \Endo[\cT]{i+j}{M}.
\]
\end{definition}

To compute the graded endomorphism ring of a classical generator, it is worth first looking at the graded endomorphism rings of three of the four types indecomposable objects; long arcs, limit arcs and double limit arcs.
We need not consider short arcs as they do not correspond to the direct summand of a minimal generator \cite{Generators}.
We follow the concept of a similar proof found in \cite{ACFGS}.

\begin{lemma}\label[lem]{Lem:EndX}
Let $X \in \ocC$ be an indecomposable object that corresponds to a limit arc.
Then 
\[
\Endo[\ocC]{\ast}{X} \cong k [x]
\]
as graded rings, with $x$ concentrated in degree $-1$.
\end{lemma}

We follow part of the proof of \cite[Prop. 3.6]{ACFGS}.

\begin{proof}
By \Cref{Prop:PY3.14}, we have 
\[
\dim_k(\Endo[\ocC]{l}{X}) =
\begin{cases*}
1 & if $l \leq 0$,\\
0 & if $l>0$,
\end{cases*}
\]
which agrees on dimensions with $k[x]$ considered as a graded ring with $x$ concentrated in degree $-1$.
All that is left to show is that the ring structures on $\Endo[\ocC]{\ast}{X}$ and $k [x]$ agree.

Let $f \in \Endo{-l}{X}$ be a morphism with $l>0$, then by the construction of $\ocC$ in \cite{Paquette2020} and using \Cref{Prop:PY3.14}, we see $f$ factors through an $l$-fold product $\Endo[\ocC]{-1}{X} \times \ldots \times \Endo[\ocC]{-1}{X}$.
This shows that the graded endomorphism ring of $X$ is isomorphic to $k [x]$ with $x$ in degree $-1$.
\end{proof} 

\begin{lemma}\label[lem]{Lem:EndY}
Let $Y \in \ocC$ be an indecomposable object that corresponds to a double limit arc.
Then 
\[
\Endo[\ocC]{\ast}{Y} \cong k [x^{\pm1}]
\]
as graded rings, with $x$ concentrated in degree $-1$.
\end{lemma}

\begin{proof}
By \Cref{Prop:PY3.14} we have
\[
\dim_k (\Endo[\ocC]{l}{Y}) = 1
\]
for all $l \in \mathbb{Z}$.
We need to show that multiplication agrees.

Given $g \in \Endo[\ocC]{-l}{Y}$ for $l > 0$, we may use the same approach as in \Cref{Lem:EndX} to show that $g$ factors through an $l$-fold product $\Endo[\ocC]{-1}{Y} \times \ldots \times \Endo[\ocC]{-1}{Y}$.
Hence 
\[
\bigoplus_{l \leq 0} \Endo[\ocC]{l}{Y} \cong k[x]
\]
such that $x$ is placed in degree $-1$.

Let $f \in \Endo{l'}{Y}$ for $l' > 0$, we show that there exists an $l'$-fold product $\Endo[\ocC]{1}{Y} \times \ldots \times \Endo[\ocC]{1}{Y}$.
Let $A \in \cC_{2n}$ such that $\pi(A) \cong Y$ under the localisation functor $\pi$.
There exists an $l$-fold product $h = h_1 \cdots h_{l'} \in \Endo[\cC_{2n}]{-1}{A} \times \ldots \times \Endo[\cC_{2n}]{-1}{A}$ in $\cC_{2n}$ by \cite{Igusa2013}, where each $h_i$ is in the multiplicative system $\Omega$.
Therefore in $\ocC$, there exists a morphism $\pi(h_i)^{-1} \in \Endo{1}{\pi(A)} \cong \Endo{1}{Y}$, and so $f=\pi(h)^{-1} \in \Endo{l'}{Y}$ factors through an $l'$-fold product $\Endo[\ocC]{1}{Y} \times \ldots \times \Endo[\ocC]{1}{Y}$.
Thus 
\[
\bigoplus_{l' \geq 0} \Endo[\ocC]{l'}{Y} \cong k [y]
\]
such that $y$ is placed in degree $1$.

Let $s \in \Endo[\ocC]{1}{Y}$, then $s =\pi(t)^{-1}$ where $t \in \Omega$, as $\cC_{2n}{A}{A[1]} = 0$ where $\pi(A) \cong Y$.
Then there exists $s' \in \Endo[\ocC]{-1}{Y}$ such that $s' = \pi(t)$, and so $ss' = id_{Y}$.
Therefore 
\[
\bigoplus_{l \in \mathbb{Z}} \Endo[\ocC]{l}{Y} \cong k[x^{\pm 1}]
\]
with $x$ placed in degree $-1$.
\end{proof}

Unfortunately, the graded endomorphism ring of a long arc is somewhat less pleasant than those of a limit arc or double limit arc.
In particular, the multiplication in the graded endomorphism ring of a long arc is seems rather unnatural in comparison.

\begin{lemma}\label[lem]{Lem: long endo}
    Let $Z \in \ocC$ be an indecomposable object corresponding to a long arc.
    Then $\Endo{\ast}{Z} \cong \mathfrak{R}$, where
    \[ 
    \mathfrak{R}^i \cong \begin{cases*}
        k & if $i \neq 1$,\\
        0 & if $i=1$.
    \end{cases*}
    \]
    For $a \in \mathfrak{R}^i$ and $b \in \mathfrak{R}^j$, multiplication is given by,
    \[
    a \cdot b = \begin{cases*}
    ab & if $i,j \leq 0$,\\
    ab & if $i \leq 0, j \geq 2$ and $j>-i$,\\
    ab & if $i \geq 2, j \leq 0$ and $-j<i$,\\
    0 & else.
    \end{cases*}
    \]
\end{lemma}

\begin{proof}
    By \Cref{Prop:PY3.14}, 
    \[ 
    \Endo{i}{Z} \cong \begin{cases*}
        k & if $i \neq 1$,\\
        0 & if $i=1$.
    \end{cases*}
    \]
    It also follows from the same result that two morphisms $f:X \rightarrow Y$ and $g : Y \rightarrow X$, with $X,Y$ indecomposable objects, compose to zero unless $Y \cong X$.
    It remains to show how multiplication works in $\Endo{\ast}{Z}$.

    Suppose that $\ell_Z = \{z_0,z_1\}$, and let $\ell_{Z[i]}=\{z_0-i,z_1-i\}$.
    Then by \Cref{Prop:PY3.14} it follows that $f \in \Endo{i}{Z}$ and $g \in \Endo{j}{Z}$ compose non-trivially to $gf \in \Endo{i+j}{Z}$ if and only if one of the following holds;
    \begin{align}
    z_0 \leq z_0-i \leq z_0-i-j < z_1,\\
    z_0 \leq z_0-j \leq z_1-i-j < z_1,\\
    z_0 \leq z_1-j \leq z_1-i-j < z_1.
    \end{align}
    It follows that in case (3.A), $i,j \leq 0$, in case (3.B), $i \leq 0,j \geq 2$ and $j>-i$, and in case (3.C) $i \geq 2, j \leq 0$ and $-j<i$.
\end{proof}

With the graded endomorphism rings of all of the relevant indecomposable objects computed, it is time for us to look at graded endomorphism rings of (some) decomposable objects.
To begin with, we look at objects with exactly two non-isomorphic indecomposable direct summands, and consider some conditions on the choice of those direct summands.

\begin{lemma}\label[lem]{Lem:accumulation point endo}
    Let $X,Y \in \ocC$ be two indecomposable objects such that the corresponding arcs are (double) limit arcs, and share an accumulation point, such that $Y \not\cong X[i]$ for all $i \in \mathbb{Z}$.
    Suppose, without loss of generality, that $Y$ is an anticlockwise rotation of $X$.
    Then 
    \[
    \Endo{\ast}{X \oplus Y} \cong \begin{pmatrix}
        \Endo{\ast}{X} & k[x^{\pm 1}]\\
        0 & \Endo{\ast}{Y}
    \end{pmatrix}
    \]
    with multiplication given by matrix multiplication, and $x$ concentrated in degree $-1$.
\end{lemma}

\begin{proof}
    The main diagonal entries of $\Endo{\ast}{X \oplus Y}$ are given by the graded endomorphism rings of both indecomposable direct summands of $X \oplus Y$, and the other two entries are given by $\Ext{\ast}{X}{Y}$ and $\Ext{\ast}{Y}{X}$ respectively.

    The graded endomorphism rings of $X$ and $Y$ respectively are given by identifying what type of arc they correspond to and apply either  \Cref{Lem:EndX} or \Cref{Lem:EndY}, whichever is appropriate.
    Given $\ell_Y$ is an anticlockwise rotation of $\ell_{X[i]}$ for all $i \in \mathbb{Z}$, then it follows from \cite{Paquette2020} that $\ocC(Y,X[i])= 0$, and so 
    \[
    \Ext{\ast}{Y}{X} = 0.
    \]
    Analogously, we find that $\ocC(Y,X[i]) \cong k$ for all $i \in \mathbb{Z}$.

    To justify our choice of $\Ext{\ast}{X}{Y} \cong k[x^{\pm 1}]$ as graded $k$-vector spaces, consider that $\Ext{\ast}{X}{Y}$ is an $\Endo{\ast}{X}-\Endo{\ast}{Y}$-bimodule, and both $\Endo{\ast}{X}$ and $\Endo{\ast}{Y}$ are isomorphic to either $k[x]$ or $k[x^{\pm 1}]$ (with $x$ concentrated in degree $-1$) by \Cref{Lem:EndX,Lem:EndY}.
    Hence $k[x^{\pm 1}]$ is a natural $\Endo{\ast}{X}-\Endo{\ast}{Y}$-bimodule, and our result follows.
\end{proof}

In this next result, we look at pairs of indecomposable objects such that they cross for any power of the suspension functor applied to either of them.
This is somewhat for convenience sake, as we then do not need to consider at which power of the suspension functor that the two indecomposable objects no longer cross.

\begin{lemma}\label[lem]{Lem:crossing endo}
    Let $X,Y \in \ocC$ be two indecomposable objects such that the corresponding arcs cross and $\mathscr{M}_X \cap \mathscr{M}_Y = \emptyset$.
    Then 
    \[
    \Endo{\ast}{X \oplus Y} \cong \begin{pmatrix}
        \Endo{\ast}{X} & k[x^{\pm 1}]\\
        k[x^{\pm 1}] & \Endo{\ast}{Y}
    \end{pmatrix}
    \]
    with multiplication given by matrix multiplication, and $f \otimes_{\Endo{\ast}{Y}} g=g \otimes_{\Endo{\ast}{X}} f=0$ for $f \in \Ext{\ast}{X}{Y}$ and $g \in \Ext{\ast}{Y}{X}$, if $X,Y$ are (double) limit arcs.
    If $X$ (resp. $Y$) is a long arc, then $f \otimes_{\Endo{\ast}{Y}} g = fg$ (resp. $g \otimes_{\Endo{\ast}{X}} f = gf$) if $\lvert fg \rvert > 0$ (resp. $\lvert gf \rvert > 0$), else multiplication is trivial.
\end{lemma}

\begin{proof}
    The main diagonal entries of $\Endo{\ast}{X \oplus Y}$ are given by the graded endomorphism rings of both indecomposable direct summands of $X \oplus Y$, and the other two entries are given by $\Ext{\ast}{X}{Y}$ and $\Ext{\ast}{Y}{X}$ respectively.
    A similar argument to \Cref{Lem:accumulation point endo} shows that $\Ext{\ast}{X}{Y}$ and $\Ext{\ast}{Y}{X}$ are both isomorphic to $k[x^{\pm 1}]$ as $k$-vector spaces.

    The multiplication rules then follow from \Cref{Lem: Factoring in Cn}.
\end{proof}

In the proof of \Cref{Lem:crossing endo}, we make a slight oversight showing that $\Ext{\ast}{X}{Y}\cong k[x^{\pm 1}]$, in that $\Ext{\ast}{X}{Y}$ may be an $\mathfrak{R}-\mathfrak{R}$ bimodule, and so our argument for why we chose $k[x^{\pm 1}]$ for our notation doesn't necessarily make much sense.
However, $\Ext{\ast}{X}{Y}$ and $k[x^{\pm 1}]$ agree on dimensions in each degree, and therefore as graded $k$-vector spaces, so we make the choice as a matter of notational convenience.

We may now state the graded endomorphism ring for a class of generators of $\ocC$.

\begin{proposition}\label[prop]{Prop: endo ring}
    Let $G \cong \bigoplus_{i=1}^m G_i \in \ocC$ be a minimal classical generator, with $G_i$ indecomposable for all $i$, such that $\mathscr{M}_{G_i} \cap \mathscr{M}_{G_j} = \emptyset$ or $\{a\}$ for some accumulation point $a$.
    Then $\Endo{\ast}{G}$ is isomorphic to a generalised $(m \times m)$-matrix $k$-algebra, labelled $\chi^G$, with entries given by,
    \[
    \chi^G_{i,j} \cong \begin{cases*}
        k[x] & if $i=j$ and $\ell_{G_i}$ is a limit arc,\\
        \mathfrak{R} & if $i=j$ and $\ell_{G_i}$ is a long arc,\\
        0 & if $i \neq j$ and, $\ell_{G_i}$ and $\ell_{G_j}$ do not cross,\\
        0 & if $i \neq j$ and, $\ell_{G_i}$ and $\ell_{G_j}$ share an accumulation point, and $\ell_{G_i}$ is an anticlockwise rotation of $\ell_{G_j}$,\\
        k[x^{\pm 1}] & else.
    \end{cases*}
    \]
    Where multiplication of $f \in \chi^G_{i,j}$ and $g = \chi^G_{j,l}$, for $f,g \neq 0$, in $\chi^G$ is given by 
    \[
    \varphi_{ijl}(f \otimes_{\Endo{}{G_j}} g) = \begin{cases*}
        gf & if $\ell_{G_i},\ell_{G_j},\ell_{G_l}$ share an accumulation point as an endpoint,\\
        gf & if $\ell_{G_i},\ell_{G_j},\ell_{G_l}$ cross and satisfy \Cref{Prop:PY3.14}, with $G_j$ the object factored through,\\
        gf & if $l=i$, $G_i$ is a long arc, and we are in the case of \Cref{Lem:crossing endo},\\
        0 & else.
    \end{cases*}
    \]
\end{proposition}

\begin{proof}
Let $\chi^G_{i,j} =\Ext{\ast}{G_i}{G_j}$.

The entry $\chi^G_{i,i}$ is given by \Cref{Lem:EndX} if $G_i$ is a limit arc, \Cref{Lem:EndY} if $G_i$ is a double limit arc, and \Cref{Lem: long endo} if $G_i$ is a long arc.

\Cref{Lem:crossing endo}, resp.\ \Cref{Lem:accumulation point endo}, gives the $(2 \times 2)$-matrix subalgebra,
\[
\begin{pmatrix}
    \chi^G_{i,i} & \chi^G_{i,j}\\
    \chi^G_{j,i} & \chi^G_{j,j}
\end{pmatrix}
\]
for all $G_i$ and $G_j$ such that $\ell_{G_i}$ and $\ell_{G_j}$ cross, resp.\ share an accumulation point.
If $\ell_{G_i}$ and $\ell_{G_j}$ do not cross, then $\Ext{\ast}{G_i}{G_j}=0= \Ext{\ast}{G_j}{G_i}$, and so $\chi^G_{i,j} = 0 = \chi^G_{j,i}$.

The multiplication follows from \Cref{Lem:accumulation point endo,Lem:crossing endo}.
\end{proof}

\subsection{Piano Algebras and Classical Generators}

We investigate the connections between the piano algebras defined for the marked surface $(D^2,M_n,\emptyset)$, where $\lvert M_{\rcirc} \rvert = n$, and the limit generators of the Paquette-Y\i ld\i r\i m category $\ocC$.
Our aim is to prove the following theorem.

\begin{theorem}\label{Thm:path algebras}
	Let $G$ be a limit generator of $\ocC$.
	Then there is an isomorphism of graded $k$-algebras,
	\[
	\chi^G\xrightarrow{\sim} kQ^G/I^G.
	\]
\end{theorem}

Recall from \Cref{Sec: Dissections} that a piano quiver associated to $(D^2,\mathscr{M})$ is the data $(Q^G,I^G,B^G)$, for some limit generator $G in \ocC$, where for every vertex in $Q_0^G$ there exists a loop labelled by $\alpha$ in degree $-1$, and for every vertex not in $B^G$, there exists a loop $\beta$ in degree 1.
We recall the ideal $I^G$ too;
\[
I^G := \langle \alpha_{a} \beta_{a}-\iota_{a},\, \beta_{a} \alpha_{a}-\iota_{a},\, \alpha_a \delta_{ab} - \delta_{ab} \alpha_b,\, \beta_{a'} \delta_{a'a} \cdots \delta_{bb'} - \delta_{a'a} \cdots \delta_{bb'} \beta_{b'}, \widehat{I^{\Delta}} \rangle,
\]

We begin with some results concerning submodules of $kQ^G/I^G$.

\begin{lemma}\label[lem]{Lem: equiv classes of paths}
	Let $a,b$ be different vertices in $Q^G_0$, and let $m \in \mathbb{Z}$.
	Then $(\iota_a(kQ^G/I^G)\iota_b)^m \cong (\iota_a(kQ^G/I^G)\iota_b)^0$ as $k$-vector spaces.
	Moreover, the dimension of $(\iota_a(kQ^G/I^G)\iota_b)^m$ as a $k$-vector space is at most 1.  
\end{lemma}

\begin{proof}
	Suppose $\sigma = \rho \delta \rho \delta \cdots \rho \delta \rho$ is a path of degree $m \in \mathbb{Z}$, starting at $a$ and ending at $b$ where each $\rho$ is some path consisting of $\alpha$ and $\beta$ for a fixed vertex.
	By the relations in $I^G$, we can commute $\alpha$ freely with $\beta$ and $\delta$.
	Hence
	\[
	\sigma \equiv \sigma' + I^G = \beta^{i_1} \delta \beta^{i_2} \delta \cdots \beta^{i_{p-1}} \delta \beta^{i_p} \alpha^i + I^G,
	\]
	where $m = \sum_{l=1}^p (i_l) - i$.
	
	Similarly, according to the relations in $I^G$, we may commute $\beta$ with paths starting and ending at vertices not in $B$.
	By \Cref{Lem: arcs between binding}, it follows that there is no arrow between vertices both in $B$, hence we may have either
	\[
	\sigma \equiv \sigma_1'' + I^G = \delta \cdots \delta\delta \alpha^i \beta^j + I^G,
	\]
	or
	\[
	\sigma \equiv \sigma_2'' + I^G = \delta \cdots \delta \beta^j \delta \alpha^i + I^G.
	\]
	However, in both $\sigma_1''$ and  $\sigma_2''$, we must have $m=j-i$, as $\sigma$ is in degree $m$ and all $\delta$ are in degree $0$.
	Then one of the following is true;
	\[
	\sigma \equiv \begin{cases*}
		\delta \cdots \delta \delta \alpha^{-m} + I^G  & if $m \leq 0$,\\
		\delta \cdots \delta \delta \beta^m + I^G  & if $m \geq 0$ and $b \not\in B$,\\
		\delta \cdots \delta \beta^m \delta + I^G & if $m \geq 0$ and $b \in B$.
	\end{cases*}
	\]
	That is, any non-zero path from $a$ to $b$ of degree $m$, is equivalent to one of the above paths in $kQ^G/I^G$, determined by $m$ and $b$.
	Indeed, such a path exists if and only if a path of degree $0$ between $a$ and $b$ exists, as \Cref{Lem: arcs between binding} implies that any path starting and ending at different vertices must pass through a vertex not in $B$.
	Hence $(\iota_a(kQ^G/I^G)\iota_b)^m \cong (\iota_a(kQ^G/I^G)\iota_b)^0$, and it follows that the dimension as a $k$-vector space is at most $1$ as $\widehat{Q^G}$ is a tree.
\end{proof}

Here we show that there is an isomorphism between the paths starting ending at a vertex $a$, and the graded endomorphism ring of the indecomposable direct summand of a limit generator in $\ocC$ associated to the vertex $a$ by \Cref{Prop: Bijection of ext admissible dissections}.

\begin{lemma}\label[lem]{Lem: loop algebras}
	There is an isomorphism of graded $k$-algebras,
	\[
	\iota_a(kQ^G/I^G)\iota_a \cong \Endo[\ocC]{\ast}{G_a},
	\]
	for any $a \in Q_0$.
\end{lemma}

\begin{proof}
	By \Cref{Lem: equiv classes of paths}, there exists at most one non-zero equivalence class of paths of degree $i$ starting and ending at vertex $a \in Q_0^G$.
	As all vertices have $\iota_a$, and paths $\alpha_a^i$, then there exists a non-zero equivalence path in degree $i \leq 0$ for all vertices $a\in Q_0^G$.
	
	If $a \in B$, then $G_a$ corresponds to a binding arc, and so is a limit arc under the bijection \Cref{Prop: Bijection of ext admissible dissections}, and so $\Endo[\ocC]{\ast}{G_a} \cong k[x]$, with $\lvert x \rvert =-1$ by \Cref{Lem:EndX}.
	There is no loop $\beta$ at $a$, and $\widehat{Q^G}$ is a tree, so there are no non-zero equivalence classes of paths of degree $i > 0$ starting and ending at $a$.
	Then a $k$-linear map $\varphi_a : \iota_a(kQ^G/I^G)\iota_a \rightarrow k[x]$ under the assignment $\alpha_a^i \mapsto x^i$ is an isomorphism as concatenation of paths is equivalent to multiplication of the induced powers of $x$.
	
	If $a \not\in B$, then $G_a$ is a double limit arc and $\Endo[\ocC]{\ast}{G_a} \cong k[x^{\pm 1}]$, with $\lvert x \rvert =-1$ by \Cref{Lem:EndY}.
	There is also a loop $\beta_a$, in degree $1$, and so there is a non-zero equivalence class of paths of degree $i$ starting and ending at $a$, for all $i \in \mathbb{Z}$.
	Then the $k$-linear map $\varphi_a : \iota_a(kQ^G/I^G)\iota_a \rightarrow k[x^{\pm 1}]$, taking $\alpha_a^i \mapsto x^i$ and $\beta_a^i \mapsto x^{-i}$, is an isomorphism of $k$-algebras, again by taking concatenation of equivalence classes of paths to multiplication of the induced powers of $x$, and considering the relations $\alpha_a\beta_a -\iota_a$ and $\beta_a \alpha_a - \iota_a$.
\end{proof}

The above result may be lifted to an isomorphism of the submodule generated by paths starting at a vertex $a$ and ending at a vertex $b$, and the graded homomorphism ring from $G_a$ to $G_b$, the respective indecomposable direct summands of the generator $G$.

\begin{lemma}\label[lem]{Lem: path modules}
	Let $a,b$ be distinct vertices in $Q_0$.
	Then there is an isomorphism of graded $k$-vector spaces
	\[
	\iota_a(kQ^G/I^G)\iota_b \cong \Ext[\ocC]{\ast}{G_a}{G_b} .
	\]
	Moreover, this lifts to an isomorphism of $\Endo{\ast}{G_a}$-$\Endo{\ast}{G_b}$-bimodules.
\end{lemma}

\begin{proof}
	If there is a non-zero morphism $f^i \in \ocC(G_a,G_b[i])$ if and only if $G_a$ and $G_b[i]$ cross or share an endpoint at an accumulation point, with $G_b[i]$ an anticlockwise rotation of $G_a$. As $G$ is a limit generator, and so $G_a$ and $G_b[i]$ do not cross, this means that $G_a$ and $G_b[i]$ must share an endpoint at an accumulation point, and $G_b[i]$ is a rotation of $G_a$ in the positive orientation of $D^2$.
	Also, as $a \neq b$, this implies there is a non-zero morphism $f^j \in \ocC(G_a,G_b[j])$ for all $j \in \mathbb{Z}$ by \Cref{Lem:accumulation point endo}.
	
	However, by the construction of $(Q^G,I^G)$, there is a path of degree $0$ from $a$ to $b$, such that no subpath is in $I^G$, hence there is a non-zero equivalence class of paths in $kQ^G/I^G$ starting at $a$ and ending at $b$ for all $i \in \mathbb{Z}$.
	By \Cref{Lem: equiv classes of paths}, there exists at most one non-zero equivalence class of paths of degree $i$ starting at vertex $a$ and ending at vertex $b$.
	
	Hence there is an isomorphism of graded $k$-vector spaces, as the dimension of $\Ext[\ocC]{\ast}{G_a}{G_b}$ and $\iota_a(kQ^G/I^G)\iota_b$ in degree $i$ is 1 for all $i \in \mathbb{Z}$.
	We denote this isomorphism taking the non-zero equivalence class of paths of degree $i$ to $x^i$ by $\varphi_{a,b}$.
	
	To see that this isomorphism lifts to an isomorphism of $\Endo{\ast}{G_a}$-$\Endo{\ast}{G_b}$-bimodules, let $\sigma_i \in (\iota_a(kQ^G/I^G)\iota_b)^i$, and observe that,
	\[
	x^q \varphi_{a,b}(\sigma_p) = x^q \cdot x^p = x^{p+q} = \varphi_{a,b} (\sigma_{p+q}),
	\]
	and 
	\[
	\varphi_{a,b} (\sigma_p) x^{q'} = x^p \cdot x^{q'} = x^{p+q'} = \varphi_{a,b} (\sigma_{p+q'}),
	\]
	for $x^q \in \Endo{\ast}{G_a}$ and $x^{q'} \in \Endo{\ast}{G_b}$, where $\iota_a(kQ^G/I^G)\iota_a \cong \Endo[\ocC]{\ast}{G_a}$ and $\iota_b(kQ^G/I^G)\iota_b \cong \Endo[\ocC]{\ast}{G_b}$ by \Cref{Lem: loop algebras}.
\end{proof}

We are now in a position to prove \Cref{Thm:path algebras}.

\begin{proof}[Proof of \Cref{Thm:path algebras}]
	\Cref{Lem: loop algebras,Lem: path modules}, and \Cref{Prop: endo ring} combine to show that there is an isomorphism of graded $k$-vector spaces $h: kQ^G/I^G \rightarrow \chi^G$.
	All that is left to show that this is an isomorphism of $k$-algebras.
	
	Let $\sigma_i \in kQ^G/I^G$ be a path from $a$ to $b$ of degree $i$, and $\sigma'_j \in kQ^G/I^G$ a path from $b$ to $c$ of degree $j$, such that $\sigma'_j\sigma_i = \sigma''_{i+j}$ is non-zero in degree $i+j$.
	Then we have
	\[
	h(\sigma'_j)h(\sigma_i) = x^j \cdot x^i = x^{i+j} = h(\sigma''_{i+j}) h(\sigma'_j\sigma_i),
	\]
	and if $\sigma'_j\sigma_i$ is zero, then it is clear that 
	\[
	h(\sigma'_j)h(\sigma_i) = 0 = h(\sigma'_j\sigma_i).
	\]
	Hence $h$ is an isomorphism of $k$-algebras.
\end{proof}

It seems reasonable that there is a natural connection between piano algebras and locally gentle algebras.
This is because piano algebras are defined via extended admissible dissections, locally gentle algebras are in bijection to admissible dissections of marked surfaces, and \Cref{Lem: extended become admissible} tells us that all extended admissible dissections may be viewed as admissible dissections.

For the case of a disc without punctures and with $2n$ marked points, we make this connection explicit.

\begin{lemma}\label[lem]{Lem: Lambda 0 is gentle}
	Let $G$ be a limit generator of $\ocC$.
	Then $(\chi^G)^0 \cong k\widehat{Q^G}/\widehat{I^G}$ as $k$-algebras, in particular, $(\chi^G)^0$ is a locally gentle algebra.
\end{lemma}

\begin{proof}
	The result follows from \Cref{Thm:path algebras}, \Cref{Prop: Bijection of ext admissible dissections} and \Cref{Cor: Extended to Gentle}.
\end{proof}

The following result tells us about the structure of $\Lambda^G$, in particular what the degree-wise components of $\Lambda^G$ look like.

\begin{proposition}\label[prop]{Prop: Piano as paths}
	Let $G \in \ocC$ be a limit generator, and let $\Lambda^G$ be its associated piano algebra.
	Then we have the following isomorphisms of $k\widehat{Q^G}/\widehat{I^G}$-modules;
	\[
	(\Lambda^G)^i \cong
	\begin{cases*}
		k\widehat{Q^G}/\widehat{I^G} & if $i \leq 0$,\\
		\bigoplus_{a \in Q_0} p_{\overline{a}} & if $i > 0$,
	\end{cases*}
	\]
	where $p_{\overline{a}}$ is the projective at $\overline{a}$, and $\overline{a} = a$ if $a \not\in B$, $\overline{a} = a'$ for an arrow $a' \rightarrow a$ if $a \in B$, and $p_{\overline{a}}=0$ if $a \in B$ and $a$ is a source.
\end{proposition}

\begin{proof}
	Let $e_a$ be the trivial path at vertex $a \in Q_0$.
	Then by \Cref{Lem: equiv classes of paths}, $(e_a\Lambda^G)^i$ has a basis given by the equivalence classes of paths in degree $i$ ending at vertex $a$.
	If $i \leq 0$, this implies a one-to-one correspondence between basis elements of $(e_a\Lambda^G)^i$ and paths in $k\widehat{Q^G}/\widehat{I^G}$ ending at $a$, hence $(e_a\Lambda^G)^i \cong p_a$.
	Therefore, as $\chi^G\cong \bigoplus_{a \in Q_0} e_a \Lambda^G$, we see that when $i \leq 0$
	\[
	 (\Lambda^G)^i \cong (\bigoplus_{a \in Q_0} e_a \Lambda^G)^i \cong \bigoplus_{a \in Q_0} (e_a \Lambda^G)^i \cong \bigoplus_{a \in Q_0} p_a \cong k\widehat{Q^G}/\widehat{I^G},
	\]
	where we recall that the differential is trivial in the second isomorphism.
	
	Now let $i >0$, and suppose that $a \not\in B$, and so there is a loop in degree 1 at vertex $a$.
	Then the above argument holds and we get $ (e_a \Lambda^G)^i \cong p_a$.
	
	However, suppose that $a \in B$, then there is no loop in degree 1 at $a$, and all paths of degree $>0$ must pass through a vertex not in $B$ so that it may pass through a loop of degree 1.
	By \Cref{Lem: arcs between binding}, there must exist an arrow from a vertex not in $B$ to a vertex in $B$, or $a$ is a source. 
	Furthermore this arrow is unique if it exists, as vertices in $B$ correspond to binding arcs, and as such only one endpoint is shared by other arcs in the extended admissible dissection.
	Label this vertex by $a'$.
	
	Again by \Cref{Lem: equiv classes of paths}, we see that now the basis of $(e_a \Lambda^G)^i$ is given by the paths ending at vertex $a'$, composed with the arrow $a' \rightarrow a$.
	If such an $a'$ does not exists, then $p_{\overline{a}}=0$.
	As such, we in fact see that $e_a \Lambda^G\cong e_{a'} \Lambda^G\cong p_{a'}$, where the second isomorphism comes from the above argument, and so our claims follows.
\end{proof}

We specialise \Cref{Prop: Piano as paths} to the case of $\Lambda_n$, which will become useful for us in \Cref{Sec: The Final Equivalences}.

\subsection{Fan Generators of \texorpdfstring{$\ocC$}{Cn}}

Let $E \in \ocC$ be the object in \Cref{fig:XandYs}, then $E$ is a generator by \Cref{Thm:Gens}, and we call $E$ a \textit{fan generator} of $\ocC$.
Fan generators are a particularly nice class of limit generators, as they exhibit some desirable properties that allow us to apply previously known results.

As a fan generator is a class of limit generator, then there exists a piano algebra $\Lambda^E$ that is isomorphic as a graded algebra to $\mathrm{End}^{\ast}_{\ocC}(E)$ by \Cref{Thm:path algebras}.
From now, when considering a fan generator and its associated piano algebra, we shall in general drop the superscript and write $\Lambda_n \cong \Lambda^E$.

A generalisation of fan generators in arbitrary triangulated categories, called \textit{linear generators}, were studied previously by the authors in \cite{LinearGens}.
In particular, it was shown that the perfect derived category of $\Lambda_n$ is equivalent as triangulated categories to the Paquette--Y\i ld\i r\i m category $\ocC$.

\begin{figure}[h!]
	\centering
	\begin{tikzpicture}
		\draw (0,0) circle (3cm);
		\draw[thin] (0,3) .. controls (0.5,2.6) .. ({3*sin(36)},{3*cos(36)}) node[pos=0.85, below] {\footnotesize $E_{2n-1}$};
		\draw[thin] (0,3) .. controls (1,0.3) .. ({3*sin(108)},{3*cos(108)}) node[pos=0.75, below] {\footnotesize $E_{2n-3}$};
		\draw[thin] (0,3) .. controls (-0.5,2.6) .. ({3*sin(324)},{3*cos(324)}) node[pos=0.85, below] {\footnotesize $E_1$};
		\draw[thin] (0,3) .. controls (-1,0.3) .. ({3*sin(252)},{3*cos(252)}) node[pos=0.75, right] {\footnotesize $\, E_3$};
		\draw[thin] (0,3) .. controls (1.5,1.5) .. ({3*sin(72)},{3*cos(72)}) node[pos=0.75, below] {\footnotesize $E_{2n-2}$};
		\draw[thin] (0,3) .. controls (0.5,0) .. ({3*sin(144)},{3*cos(144)}) node[pos=0.85, left] {\footnotesize $E_{2n-4}$};
		\draw[thin] (0,3) .. controls (-1.5,1.5) .. ({3*sin(288)},{3*cos(288)}) node[pos=0.75, below] {\footnotesize $E_2$};
		\draw[thin] (0,3) .. controls (-0.5,0) .. ({3*sin(216)},{3*cos(216)}) node[pos=0.85, right] {\footnotesize $E_4$};
		\draw[thick,dotted] ({2.5*sin(190)},{2.5*cos(190)}) -- ({2.5*sin(170)},{2.5*cos(170)});
		\path[fill=white] ({3*sin(190)},{3*cos(190)}) -- ({3.1*sin(190)},{3.1*cos(190)}) -- ({3.1*sin(170)},{3.1*cos(170)}) -- ({3*sin(170)},{3*cos(170)}) -- cycle;
		\node at ({3.4*sin(0)},{3.4*cos(0)}) {$\mathfrak{a}_1$};
		\node at ({3.4*sin(288)},{3.4*cos(288)}) {$\mathfrak{a}_2$};
		\node at ({3.4*sin(216)},{3.4*cos(216)}) {$\mathfrak{a}_3$};
		\node at ({3.4*sin(144)},{3.4*cos(144)}) {$\mathfrak{a}_{n-1}$};
		\node at ({3.4*sin(72)},{3.4*cos(72)}) {$\mathfrak{a}_n$};
		\node at ({3.4*sin(324)},{3.4*cos(324)}) {$z_1$};
		\node at ({3.4*sin(252)},{3.4*cos(252)}) {$z_2$};
		\node at ({3.4*sin(36)},{3.4*cos(36)}) {$z_n$};
		\node at ({3.5*sin(108)},{3.5*cos(108)}) {$z_{n-1}$};
		\draw[fill=white] (0,3) circle (0.1cm);
		\draw[fill=white] ({3*sin(144)},{3*cos(144)}) circle (0.1cm);
		\draw[fill=white] ({3*sin(72)},{3*cos(72)}) circle (0.1cm);
		\draw[fill=white] ({3*sin(288)},{3*cos(288)}) circle (0.1cm);
		\draw[fill=white] ({3*sin(216)},{3*cos(216)}) circle (0.1cm);
	\end{tikzpicture}
	\caption{A \textit{fan generator} of $\ocC$.}
	\label[fig]{fig:XandYs}
\end{figure}

The following corollaries are direct applications of \Cref{Prop: endo ring} and \Cref{Prop: Piano as paths} respectively.

\begin{corollary}\label[cor]{Cor: Fan Generator}
	Let $E$ be a fan generator of $\ocC$.
	Then 
	\[
	(\chi^E)_{i,j} \cong \begin{cases*}
		k[x] & if $i=j$ is odd,\\
		0 &  if $i > j$,\\
		k[x^{\pm 1}] & else.
	\end{cases*}
	\]
\end{corollary}

\begin{corollary}\label[cor]{Cor: Lambda as A modules}
	Let $\Lambda_n$ be the graded algebra associated to the fan generators of $\ocC$, and let $E \in \ocC$ be a fan generator.
	Then $(\Lambda_n)^0 \cong kA_{2n-1}$ as $k$-algebras, where $A_{2n-1}$ is the linearly orientated quiver of Dynkin type $A$ with $2n-1$ vertices.
	Moreover;
	\[
	(\Lambda_n)^i \cong \begin{cases*}
		kA_{2n-1} & if $i \leq 0$,\\
		\bigoplus_{a=1}^{n-1} p_{2a}^{\oplus 2} & if $i > 0$,
	\end{cases*}
	\]
	as $kA_{2n-1}$-modules, where $p_b$ is the projective $kA_{2n-1}$-module at vertex $b$.
\end{corollary}

\begin{proof}
	This follows from \Cref{Prop: Piano as paths}, and the piano quiver $(Q,I,B)$ associated to $\Lambda_n$ having the Dynkin quiver $A_{2n-1}$ as its underlying keyboard quiver.
\end{proof}

\begin{example}
	Let $n=3$, then by \Cref{Cor: Fan Generator} the graded matrix algebra $\Lambda_3$ would be 
	\[
	\Lambda_3 =
	\begin{pmatrix}
		k [x] & k [x^{\pm 1}] & k [x^{\pm 1}] & k [x^{\pm 1}] & k [x^{\pm 1}]\\
		0 &  k [x^{\pm 1}] & k [x^{\pm 1}] & k [x^{\pm 1}] & k [x^{\pm 1}]\\
		0 & 0 & k [x] & k [x^{\pm 1}] & k [x^{\pm 1}]\\
		0 & 0 & 0 & k [x^{\pm 1}] & k [x^{\pm 1}]\\
		0& 0 & 0 & 0 & k [x]
	\end{pmatrix}
	\]
	and the additive group $\Lambda_3^i$, for $i \leq 0$, would be
	\[
	(\Lambda_3)^i =
	\begin{pmatrix}
		k \cdot x^{-i} & k \cdot x^{-i} & k \cdot x^{-i} & k \cdot x^{-i} & k \cdot x^{-i}\\
		0 & k \cdot x^{-i} & k \cdot x^{-i} & k \cdot x^{-i} & k \cdot x^{-i}\\
		0 & 0 & k \cdot x^{-i} & k \cdot x^{-i} & k \cdot x^{-i}\\
		0 & 0 & 0 & k \cdot x^{-i} & k \cdot x^{-i}\\
		0& 0 & 0 & 0 & k \cdot x^{-i}
	\end{pmatrix}
	\]
	and for $i >0$ would be
	\[
	(\Lambda_3)^i =
	\begin{pmatrix}
		0 & k \cdot x^{-i} & k \cdot x^{-i} & k \cdot x^{-i} & k \cdot x^{-i}\\
		0 & k \cdot x^{-i} & k \cdot x^{-i} & k \cdot x^{-i} & k \cdot x^{-i}\\
		0 & 0 & 0 & k \cdot x^{-i} & k \cdot x^{-i}\\
		0 & 0 & 0 & k\cdot x^{-i} & k \cdot x^{-i}\\
		0& 0 & 0 & 0 & 0
	\end{pmatrix}
	\]
\end{example}

\begin{example}\label[ex]{Ex: Lambda quiver}
	Let $E$ be a fan generator of $\ocC$, that is, a generator with graded endomorphism ring $\Lambda_n$.
	Then the piano quiver $Q^E$ is,
	\[
	\begin{tikzcd}
		Q^E : & \bullet \arrow[r,"\delta"] \arrow[loop right, in=290,out=250,looseness=15,"\alpha",swap] & \circ \arrow[r,"\delta"] \arrow[loop right, in=290,out=250,looseness=15,"\alpha",swap] \arrow[loop,in=70,out=110,looseness=12,"\beta"] & \bullet \arrow[r,"\delta"] \arrow[loop right, in=290,out=250,looseness=15,"\alpha",swap] & \circ \arrow[loop right, in=290,out=250,looseness=15,"\alpha",swap] \arrow[loop,in=70,out=110,looseness=12,"\beta"] \arrow[r] & \cdots \arrow[r] & \circ \arrow[loop right, in=290,out=250,looseness=15,"\alpha",swap] \arrow[loop,in=70,out=110,looseness=12,"\beta"] \arrow[r,"\delta"] & \bullet \arrow[loop right, in=290,out=250,looseness=15,"\alpha",swap]
	\end{tikzcd}
	\]
	and the ideal $I^E$ is given by 
	\[
	I^E := \langle \alpha_a \beta_a  - \iota_a,\, \beta_a \alpha_a  - \iota_a,\, \alpha_a \sigma - \sigma \alpha_b,\, \beta_a  \sigma - \sigma \beta_b \rangle,
	\]
	where $a,b \in Q_0$ and $\sigma$ is any path such that concatenation makes sense.
	Then by \Cref{Thm:path algebras}, we have $\chi^E \cong kQ^E/I^E$.
\end{example}

\section{Derived Equivalence of Piano Algebras from Discs}\label{Sec: The Final Equivalences}

We use the results of the previous subsection to show that there exists a derived equivalence between piano algebras coming from the marked surface $(D^2,M_n,\emptyset)$.
In particular, we note the importance of the functor $\mathscr{F}$ preserving triangles with at least two indecomposable terms.
This is because the classification of classical generators in $\ocC$ \cite{Generators} relies solely on triangles with at least two indecomposable terms, via the homologically connected property.
Therefore, the functor $\mathscr{F}^{-1}$ preserves the property of classical generation, as so we have a description of (not necessarily all) classical generators in $\mathrm{perf}\Lambda_n$.

The following result is an equivalent restatement of \cite[Prop. 2.26]{LinearGens}.

\begin{lemma}\label[lem]{Lem: Forwards and Backwards}
	Let $f,f'$ be morphisms between indecomposable objects in $\mathrm{perf}\Lambda_n$ such that $ff' \neq 0$.
	If both $f$ and $f'$ are forward morphisms, then $ff'$ is a forwards morphism.
	Alternatively, if one of $f$ and $f'$ is a forward morphism and the other is a backwards morphism, then $ff'$ is a backwards morphism.
\end{lemma}

Before we can prove the derived equivalence of piano algebras coming from $(D^2,M_n, \emptyset)$, we must first collect some important results.

\begin{lemma}\cite{KellerDeriving}\label[lem]{Lem: AB tri equivalence}
	Let $R$ be a ring.
	Let $(A,d)$ and $(B,d)$ be dg $R$-algebras, and let $N$ be a dg $(A,B)$-bimodule.
	Assume that
	\begin{enumerate}
		\item $N$ is a compact object in $D(B,d)$,
		\item if $N' \in D(B,d)$, and $D(B,d)(N,N'[i])=0$ for all $i \in \mathbb{Z}$, then $N'=0$,
		\item the map $H^i(A,d) \rightarrow D(B,d)(N,N[i])$ is an isomorphism for all $i \in \mathbb{Z}$.
	\end{enumerate}
	Then 
	\[
	- \otimes_A^{\mathbb{L}} N : D(A,d) \rightarrow D(B,d)
	\]
	gives an $R$-linear equivalence of triangulated categories.
\end{lemma}

Here $- \otimes_A^{\mathbb{L}} N$ is the derived tensor functor.

\begin{definition}
	Let $R$ be a ring, and $(A,d)$ a dg $R$-algebra.
	Let $L,M$ be dg $R$-modules, and $\overline{L},\overline{M}$ their respective underlying graded $A$-modules, then we let
	\[
	\mathrm{Mod}^{dg}(A,d)(L,M)= \bigoplus_{n \in \mathbb{Z}} (\mathrm{Mod}A)^n(\overline{L},\overline{M}),
	\]
	where $(\mathrm{Mod}A)^n(\overline{L},\overline{M})$ is the set of $A$-module morphisms of homogeneous degree $n$.
	We equip $\mathrm{Mod}^{dg}(A,d)(L,M)$ with a differential $d$, given by 
	\[
	d(f) = d_M \circ f - (-1)^nf \circ d_L,
	\]
	for some element $f \in (\mathrm{Mod}A)^n(\overline{L},\overline{M})$.
\end{definition}

We may use the following characterisation of dg bimodule structures on a module.

\begin{lemma}[Stacks Proj 22.28.2]\label[lem]{Lem: AB bimodule}
	Let $R$ be a ring.
	Let $(A,d)$ and $(B,d)$ be dg $R$-algebras, and let $N$ be a right dg $B$-module.
	There is a 1-to-1 correspondence between $(A,B)$-bimodule structures on $N$ compatible with the given dg $B$-module structure and homomorphisms
	\[
	A \rightarrow \mathrm{Mod}^{dg}(B,d)(N,N)
	\]
	of dg $R$-algebras.
\end{lemma}

\begin{lemma}\label[lem]{Lem: moddg structure}
	Let $X \in \mathrm{perf}\Lambda_n$, and suppose that $X \cong \mathrm{cone}g_X$ for some morphism $g_X : Q \rightarrow P$ with $Q,P \in \lang[1]{\Lambda_n}$.
	Then
	\[
	\mathrm{Mod}^{dg}\Lambda_n(X,X) \cong \bigoplus_{i \in \mathbb{Z}} \begin{pmatrix}
		(\mathrm{Mod}\Lambda_n)^i(Q[1],Q[1]) & (\mathrm{Mod}\Lambda_n)^i(Q[1],P)\\
		(\mathrm{Mod}\Lambda_n)^i(P,Q[1]) & (\mathrm{Mod}\Lambda_n)^i(P,P)
	\end{pmatrix},
	\]
	with differential, for some $f \in (\mathrm{Mod}\Lambda_n)^i(X,X)$, given by
	\[
	d(f) = \begin{pmatrix}
		g^X \circ f\, \vline_P^Q & g^X \circ f\, \vline_P^P - (-1)^i f\, \vline_Q^Q \circ g^X\\
		0 & f \, \vline_P^Q \circ g^X
	\end{pmatrix},
	\]
	where $f\, \vline_A^B$ is the restriction of $f$ to $(\mathrm{Mod}\Lambda_n)^i(A,B)$, with $A,B \in \{Q[1],P\}$.
\end{lemma}

\begin{proof}
	By \cite[Theorem 2.16]{LinearGens}, it is reasonable to assume that such a morphism $g^X$ exists for any $X \in \mathrm{perf}\Lambda_n$.
	Therefore we may view $X$ as the dg $\Lambda_n$-module $\biggl (Q[1] \oplus P, \begin{pmatrix}
		0 & g^X\\
		0 & 0
	\end{pmatrix} \biggr )$.
	
	The differential then follows by the definition of the differential of $\mathrm{Mod}^{dg}\Lambda_n(X,X)$ and the differential of $X$.
\end{proof}

\begin{definition}\label[def]{Def: Signed matrix}
	Let $G \cong \bigoplus_{j=1}^{2n-1} G_j \in \mathrm{perf}\Lambda_n$ be a limit generator, and let $(\widehat{Q^G},\widehat{I^G},B^G)$ be its associated keyboard quiver.
	Let $1 \leq m \leq 2n-1$ such that $G_j \in \lang[1]{\Lambda_n}$ if $j > m$, and $G_j \not\in \lang[1]{\Lambda_n}$ if $j \leq m$.
	
	We construct $M_G$, a \textit{signed diagonal matrix associated to} $G$, as follows;
	\begin{enumerate}
		\item Take the quiver $\widehat{Q^G}$, and at vertex $j$ place the pair $(\beta_j,\delta_j)$, or $(0,\delta_j)$ if $j > m$, and choose one element in the set $\{\beta_1, \ldots, \beta_m, \delta_i, \ldots, \delta_{2n-1}\}$.
		\item If our initial choice is $\beta_j$, and there is an arrow between $j$ and $l$ in $\widehat{Q^G}$, then we choose $\beta_l$ (or $0$ if $l>m$) if the arrow corresponds to a forward morphism in $\mathrm{End}G$, and choose $\delta_l$ if it corresponds to a backwards morphism. If our initial choice is $\delta_j$, we perform the same operation, choosing $\delta_l$ if the arrow corresponds to a forward morphism, and $\beta_l$ if it corresponds to a backwards morphism.
		\item We repeat this process until we have made a choice for all $1 \leq j \leq  2n-1$, and label our choice $\alpha_j$ for each $j$.
		\item Let $M_G$ be a $(2n+m-1) \times (2n+m-1)$-diagonal matrix, such $(M_G)_{jj} = \beta_j$ if $j \leq m$, and $(M_G)_{jj} = \delta_{j-m}$ if $j > m$.
		\item Finally, we let each choice $\alpha_j =-e$, and let every other entry on the diagonal be $e$.
	\end{enumerate}
\end{definition}

\begin{example}
	Suppose we have the following limit generator $G$ of $\mathrm{perf}\Lambda_4$;
	
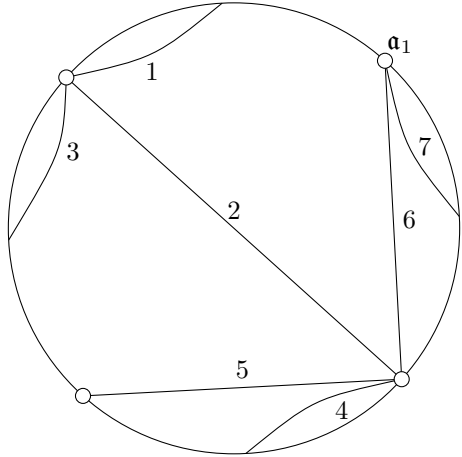
\[
	\begin{tikzpicture}
	\draw ({3*sin(42)}, {3*cos(42)}) -- ({3*sin(132)}, {3*cos(132)}) node [pos=0.5, right] {$6$};;
	\draw ({3*sin(312)}, {3*cos(312)}) -- ({3*sin(132)}, {3*cos(132)}) node [pos=0.5, above] {$2$};;
	\draw ({3*sin(222)}, {3*cos(222)}) -- ({3*sin(132)}, {3*cos(132)}) node [pos=0.5, above] {$5$};;
	\draw ({3*sin(42)}, {3*cos(42)}) .. controls({2.5*sin(65)}, {2.5*cos(65)}) ..  ({3*sin(87)}, {3*cos(87)}) node [pos=0.5, right] {$7$};;
	\draw ({3*sin(132)}, {3*cos(132)}) .. controls({2.5*sin(155)}, {2.5*cos(155)}) ..  ({3*sin(177)}, {3*cos(177)}) node [pos=0.3, below] {$4$};;
	\draw ({3*sin(312)}, {3*cos(312)}) .. controls({2.5*sin(295)}, {2.5*cos(295)}) ..  ({3*sin(267)}, {3*cos(267)}) node [pos=0.5, right] {$3$};;
	\draw ({3*sin(312)}, {3*cos(312)}) .. controls({2.5*sin(335)}, {2.5*cos(335)}) ..  ({3*sin(357)}, {3*cos(357)}) node [pos=0.5, below] {$1$};
	\disco{42}{4}{1,2,3,4}{3}
	\node at ({3.3*sin(42)},{3.3*cos(42)}) {$\mathfrak{a}_1$};
	\end{tikzpicture}
	\]
	Let $G_j$ be the indecomposable direct summand of $G$ corresponding to the arc labelled $j$, and we see that $G_6,G_7 \in \lang[1]{\Lambda_4}$, hence $m=5$.
	Following the algorithm of \Cref{Def: Signed matrix}, we have the following quiver;
	\[
	\begin{tikzcd}
		&&&& (\beta_1,\delta_1) &&&&\\
		(\beta_4,\delta_4) && (\beta_5,\delta_5) \arrow[rr] \arrow[ll] && (\beta_2,\delta_2) \arrow[u,"B"] \arrow[rr] && (0,\delta_6) \arrow[rr] && (0,\delta_7)\\
		&&&& (\beta_3,\delta_3) \arrow[u] &&&&
	\end{tikzcd}
	\]
	where the arrows corresponding to backwards morphism are denoted with a $B$.
	Suppose we make $\beta_5$ our initial choice, then we have 
	\[
	(\alpha_1,\alpha_2,\alpha_3,\alpha_4,\alpha_5,\alpha_6,\alpha_7) = (\delta_1, \beta_2, \beta_3, \beta_4, \beta_5, 0, 0).
	\]
	Therefore we get a signed diagonal matrix $M_G$;
	\[
	M_G = \left( \begin{array}{ccccc:ccccccc}
		1 & 0 & 0 & 0 & 0 & 0 & 0 & 0 & 0 & 0 & 0 & 0\\
		0 & -1 & 0 & 0 & 0 & 0 & 0 & 0 & 0 & 0 & 0 & 0\\
		0 & 0 & -1 & 0 & 0 & 0 & 0 & 0 & 0 & 0 & 0 & 0\\
		0 & 0 & 0 & -1 & 0 & 0 & 0 & 0 & 0 & 0 & 0 & 0\\
		0 & 0 & 0 & 0 & -1 & 0 & 0 & 0 & 0 & 0 & 0 & 0\\ \hdashline[2pt/2pt]
		0 & 0 & 0 & 0 & 0 & -1 & 0 & 0 & 0 & 0 & 0 & 0\\
		0 & 0 & 0 & 0 & 0 & 0 & 1 & 0 & 0 & 0 & 0 & 0\\
		0 & 0 & 0 & 0 & 0 & 0 & 0 & 1 & 0 & 0 & 0 & 0\\
		0 & 0 & 0 & 0 & 0 & 0 & 0 & 0 & 1 & 0 & 0 & 0\\
		0 & 0 & 0 & 0 & 0 & 0 & 0 & 0 & 0 & 1 & 0 & 0\\
		0 & 0 & 0 & 0 & 0 & 0 & 0 & 0 & 0 & 0 & 1 & 0\\
		0 & 0 & 0 & 0 & 0 & 0 & 0 & 0 & 0 & 0 & 0 & 1\\
	\end{array}\right) 
	\]
	where the upper left block contains $\beta_1,\ldots,\beta_5$, and the lower right block contains $\delta_1,\ldots, \delta_7$.
\end{example}

\begin{lemma}\label[lem]{Lem: beta and delta}
	Let $G \cong \bigoplus_{j=1}^{2n-1} G_j \in \mathrm{perf}\Lambda_n$ be a limit generator, and let $M^G$ be a signed diagonal matrix of $G$.
	If $G_j \rightarrow G_l$ is a forward morphism, then $\beta_j = \beta_l$, and $\delta_j = \delta_l$.
	Else, if $G_j \rightarrow G_l$ is a backwards morphism, then $\beta_j = \delta_l$ and $\delta_j = \beta_l$.
\end{lemma}

\begin{proof}
	A forward morphism $G_j \rightarrow G_l$ is a composition of forward morphisms by \Cref{Lem: Forwards and Backwards} and \cite[Prop. 2.26]{LinearGens}
	Hence the choice at $j$ is the same as the choice at $l$, as forward morphisms do not change the choice, and so $\beta_j = \beta_l$, and $\delta_j = \delta_l$.
	
	A backwards morphism $G_j \rightarrow G_l$ is a composition of forward morphisms and a single backwards morphism, again by \Cref{Lem: Forwards and Backwards} and \cite[Prop. 2.26]{LinearGens}.
	Hence the choice at $j$ and at $l$ differ, as backwards morphisms do change the choice, and so $\beta_j = \delta_l$, and $\delta_j = \beta_l$.
\end{proof}

Finally, we may prove the derived equivalence between piano algebras associated to the surface $(D^2,M_n,\emptyset)$.

\begin{theorem}\label{Thm: Derived Equiv}
	Let $G \in \mathrm{perf}\Lambda_n$ be a limit generator.
	Then $G$ is a dg $(\Lambda^G,\Lambda_n)$-bimodule.
	Moreover, there is a equivalence of triangulated categories
	\[
	D(\Lambda^G) \xlongrightarrow{\sim} D(\Lambda_n).
	\]
\end{theorem}

\begin{proof}
	Let $G \cong \bigoplus_{j=1}^{2n-1} G_i$, and suppose that $G_j \cong \mathrm{cone}g_j$ with $g_j : Q_j \rightarrow P_j$ such that $Q_j,P_j$ are indecomposable in $\lang[1]{\Lambda_n}$.
	If $G_j \in \lang[1]{\Lambda_n}$, then we let $P_j = G_j$ and $Q_j =0$.
	Let $Q = \bigoplus_{j =1}^m Q_j$ and $P = \bigoplus_{j=1}^{2n-1} P_j$, where $m \leq 2n-1$.
	We have that $G \cong \mathrm{cone}(g)$, where $g = \bigoplus_{j=1}^{2n-1} g_j$
	Then by \Cref{Lem: moddg structure} we have 
	\[
	\mathrm{Mod}^{dg}\Lambda_n(G,G) \cong \bigoplus_{i \in \mathbb{Z}} \begin{pmatrix}
		(\mathrm{Mod}\Lambda_n)^i(Q[1],Q[1]) & (\mathrm{Mod}\Lambda_n)^i(Q[1],P)\\
		(\mathrm{Mod}\Lambda_n)^i(P,Q[1]) & (\mathrm{Mod}\Lambda_n)^i(P,P)
	\end{pmatrix},
	\]
	and differential
	\[
	d(x) = \begin{pmatrix}
		g \circ x\, \vline_P^Q & g \circ x\, \vline_P^P - (-1)^i x\, \vline_Q^Q \circ g\\
		0 & x \, \vline_P^Q \circ g
	\end{pmatrix}.
	\]

	We show that there exists a dg $k$-algebra homomorphism $\Lambda^G \rightarrow \mathrm{Mod}^{dg}\Lambda_n(G,G)$.
	By \Cref{Lem: AB bimodule}, this is equivalent to showing that there exists a dg $(\Lambda^G,\Lambda_n)$-bimodule structure on $G$.
	
	We define $\pi \colon \Lambda^G \rightarrow \mathrm{Mod}^{dg}\Lambda_n(G,G)$ on homogeneous elements of $\Lambda^G$;
	\[
		x^i_{jl} \mapsto \begin{cases*}
			\begin{pmatrix}
				y^i_{jl} & 0\\
				0 & z^i_{jl}
			\end{pmatrix} \in (\mathrm{Mod}\Lambda_n)^i(Q_j[1] \oplus P_j, Q_l[1] \oplus P_l), & if $x^i_{jl}$ is a forward morphism,\\
			\begin{pmatrix}
				0 & w^i_{jl}\\
				0 & 0 
			\end{pmatrix} \in (\mathrm{Mod}\Lambda_n)^i(Q_j[1] \oplus P_j, Q_l[1] \oplus P_l), & if $x^i_{jl}$ is a backwards morphism.
		\end{cases*}
	\]
	Here, $x^i_{jl}$ is a non-zero homogeneous morphism of degree $i$ from $G_j$ to $G_l$.
	Note that by \Cref{Thm:path algebras} we have $x^i_{jl} \in k$, and we let $x^i_{jl} = y^i_{jl} = z^i_{jl}$, and $x^i_{jl} = w^i_{jl}$ whenever each is defined.
	That is, we only use $y^i_{jl}$, $w^i_{jl}$, and $z^i_{jl}$ to differentiate by position in the matrix.
	It is clear that $\pi$ is not a homomorphism of dg $k$-algebras, as it cannot commute with the differentials, given there is no $(-1)^i$ term, and we must have $d(\pi(x))=0$.
	However, we may construct a dg $k$-algebra homomorphism  from $\pi$.
	
	Let $M_G$ be a signed diagonal matrix of $G$.
	Then we define $\varphi \colon \Lambda^G \rightarrow \mathrm{Mod}^{dg}(\Lambda_n)(G,G)$ to be the map such that $\varphi (x) = (M_G)^i \cdot \pi(x)$ for $x \in (\mathrm{Mod}\Lambda_n)^i (G,G)$, where $(M_G)^i$ is the $i$-fold product of $M^G$.
	That is $\varphi$ acts on homogeneous elements of $\Lambda^G$ in the following way;
	\[
	x^i_{jl} \mapsto \begin{cases*}
		\begin{pmatrix}
			\beta_j y^i_{jl} & 0\\
			0 & \delta_j z^i_{jl}
		\end{pmatrix} \in (\mathrm{Mod}\Lambda_n)^i(Q_j[1] \oplus P_j, Q_l[1] \oplus P_l), & if $x^i_{jl}$ is a forward morphism,\\
		\begin{pmatrix}
			0 & \beta_j w^i_{jl}\\
			0 & 0 
		\end{pmatrix} \in (\mathrm{Mod}\Lambda_n)^i(Q_j[1] \oplus P_j, Q_l[1] \oplus P_l), & if $x^i_{jl}$ is a backwards morphism.
	\end{cases*}
	\]
	We claim that $\varphi$ is a dg $k$-module homomorphism.
	
	It is clear that $\varphi$ preserves scalar multiplication and addition, so it is left to show that it preserves multiplication and commutes with the differential.
	
	We must have that $d(\varphi(x))=0$, as the differential of $\Lambda^G$ is trivial.
	By \Cref{Lem: moddg structure}, we have
	\[
	d(\varphi(x)) = \begin{pmatrix}
		g \circ \varphi(x)\, \vline_P^Q & g \circ \varphi(x)\, \vline_P^P - (-1)^i \varphi(x) \, \vline_Q^Q \circ g\\
		0 & \varphi(x) \, \vline_P^Q \circ g
	\end{pmatrix}.
	\]
	However, by construction of $\varphi$, we automatically have $\varphi(x)\, \vline_P^Q=0$, as $\pi(x) \cap \mathrm{Mod}\Lambda_n(P,Q) = 0$, and $M^G$ only affects the sign of a morphism in $\pi(x)$.
	It is left to show that $g \circ \varphi(x)\, \vline_P^P - (-1)^i \varphi(x) \, \vline_Q^Q \circ g=0$.
	
	This time by the construction of $\pi$, we have that $g \circ \pi(x) \vline_Q^Q = \pi(x) \vline_P^P \circ g$.
	However, as we have $\beta_j = (-1)^i \delta_j$, then $g \circ \varphi(x) \vline_Q^Q = (-1)^i \varphi(x) \vline_P^P \circ g$,
	and so $g \circ \varphi(x)\, \vline_P^P - (-1)^i \varphi(x) \, \vline_Q^Q \circ g=0$.
	Therefore $\varphi$ commutes with the differential.
	
	Finally, we check that $\varphi$ respects multiplication.
	Let $x \in (\mathrm{Mod}\Lambda_n)^i(G,G)$ and $x' \in (\mathrm{Mod}\Lambda_n)^{i'}(G,G)$.
	Let $1 \leq j,l \leq 2n-1$, then $(xx')_{jl} = \sum_{1 \leq  j' \leq 2n-1} x_{jj'}x'_{j'l}$, such that the morphism $G_j \rightarrow G_l$ factors through $G_{j'}$
	\[
	\varphi(xx')_{jl} = \begin{cases*}
		\begin{pmatrix}
			\beta_{j}^{i+i'} \sum_{1 \leq  j' \leq 2n-1} x_{jj'}x'_{j'l} & 0\\
			0 & \delta_j^{i+i'} \sum_{1 \leq  j' \leq 2n-1} x_{jj'}x'_{j'l}
		\end{pmatrix}, & if $G_j \rightarrow G_l$ is a forward morphism,\\
		\begin{pmatrix}
			0 & \beta_j^{i+i'} \sum_{1 \leq  j' \leq 2n-1} x_{jj'}x'_{j'l}\\
			0 & 0 
		\end{pmatrix}, & if $G_j \rightarrow G_l$ is a backwards morphism.
	\end{cases*}
	\]
	We may also see that 
	\[
	(\varphi(x)\varphi(x'))_{jl} = \begin{pmatrix}
		\sum_{1 \leq  j' \leq 2n-1} (\beta_j^i y_{jj'}) (\beta_{j'}^{i'} y_{j'l}) & \sum_{1 \leq  j' \leq 2n-1} (\beta_j^i y_{jj'}) (\beta_{j'}^{i'} w_{j'l}) + \sum_{1 \leq  j' \leq 2n-1} (\beta_j^i w_{jj'}) (\delta_{j'}^{i'} z_{j'l})\\
		0 & \sum_{1 \leq  j' \leq 2n-1} (\delta_j^i z_{jj'}) (\delta_{j'}^{i'} z_{j'l}).
	\end{pmatrix}
	\]
	
	Suppose that $G_j \rightarrow G_l$ is a forward morphism, then as $y_{jj'}w{j'l} = w_{jj'}z_{j'l}$ due to $E$ being a linear generator \cite{LinearGens}, and the definition of a linear generator.
	Hence $(\beta_j^i y_{jj'}) (\beta_{j'}^{i'} w_{j'l}) = - (\beta_j^i w_{jj'}) (\delta_{j'}^{i'} z_{j'l})$, and so $\sum_{1 \leq  j' \leq 2n-1} (\beta_j^i y_{jj'}) (\beta_{j'}^{i'} w_{j'l}) + \sum_{1 \leq  j' \leq 2n-1} (\beta_j^i w_{jj'}) (\delta_{j'}^{i'} z_{j'l})=0$.
	Therefore we wish to show that;
	\begin{align*}
		\sum_{1 \leq  j' \leq 2n-1} (\beta_j^i y_{jj'}) (\beta_{j'}^{i'} y_{j'l}) &= \beta_{j}^{i+i'} \sum_{1 \leq  j' \leq 2n-1} x_{jj'}x'_{j'l}\\
		\sum_{1 \leq  j' \leq 2n-1} (\delta_j^i z_{jj'}) (\delta_{j'}^{i'} z_{j'l}) &= \delta_j^{i+i'} \sum_{1 \leq  j' \leq 2n-1} x_{jj'}x'_{j'l}.
	\end{align*}
	Clearly $y_{jj'}y_{j'l} = x_{jj'}x_{j'l} = z_{jj'}z_{j'l}$, so it remains to prove that $\beta_j^{i+i'} = \beta_j^i \beta_{j'}^{i'}$ when $y_{jj'}y_{j'l} \neq 0$.
	By \Cref{Lem: Forwards and Backwards} and \cite[Prop. 2.26]{LinearGens}, we see that $y_{jj'}$ and $y_{j'l}$ must both be forwards morphisms, and so $\beta_j = \beta_l$ by \Cref{Lem: beta and delta}.
	Hence $\beta_j^{i+i'} = \beta_j^i \beta_{j'}^{i'}$, and analogously $\delta_j^{i+i'} = \delta_j^i \delta_{j'}^{i'}$.
	Therefore we have shown the equalities.
	
	Now suppose that $G_j \rightarrow G_l$ is a backwards morphism, then $\sum_{1 \leq  j' \leq 2n-1} (\beta_j^i y_{jj'}) (\beta_{j'}^{i'} y_{j'l}) = \sum_{1 \leq  j' \leq 2n-1} (\delta_j^i z_{jj'}) (\delta_{j'}^{i'} z_{j'l})=0$, as both $y$ and $z$ are always forward morphisms, and so do not compose to a backwards morphism by \Cref{Lem: Forwards and Backwards}.
	
	Therefore we wish to show that;
	\[
	(\beta_j^i y_{jj'}) (\beta_{j'}^{i'} w_{j'l}) + (\beta_j^i w_{jj'}) (\delta_{j'}^{i'} z_{j'l})= \beta_j^{i+i'} x_{jj'}x'_{j'l}.
	\]
	It is evident that, if $x_{jj'}$ is a backwards morphism, then by \Cref{Lem: Forwards and Backwards} $x_{j'l}$ is a forwards morphism, and so $x_{jj'}x_{j'l} = w_{jj'}z_{j'l}$.
	In this case $\beta_j^i \delta_{j'}^{i'} = \beta_j^i \beta_j^{i'} = \beta_j^{i+i'}$, where the first equality holds by \Cref{Lem: beta and delta}.
	So we get
	\[
	(\beta_j^i w_{jj'}) (\delta_{j'}^{i'} z_{j'l})= \beta_j^{i+i'} x_{jj'}x'_{j'l}.
	\]
	Now, if $x_{jj'}$ is a forward morphism, then $x_{j'l}$ is a backwards morphism, again by \Cref{Lem: Forwards and Backwards}, and so $x_{jj'}x_{j'l} = y_{jj'}w_{j'l}$.
	Then we have $\beta_j^i \beta_{j'}^{i'} = \beta_j^i \beta_j^{i'} = \beta_j^{i+i'}$, where the first equality holds again by \Cref{Lem: beta and delta}.
	So we get
	\[
	(\beta_j^i y_{jj'}) (\beta_{j'}^{i'} w_{j'l})= \beta_j^{i+i'} x_{jj'}x'_{j'l},
	\]
	therefore we have shown the desired equality, and subsequently that $\varphi(xx') = \varphi(x) \varphi(x')$.
	Hence we have shown that $\varphi : \Lambda^G \rightarrow \mathrm{Mod}^{dg}(\Lambda_n)(G,G)$ is a homomorphism of dg $k$-algebras, and so there  exists a dg $(\Lambda^G,\Lambda_n)$-bimodule structure on $G$ by \Cref{Lem: AB bimodule}.
	
	Also, we know that $G$ is a compact object in $D(\Lambda_n)$, and a classical generator of $\mathrm{perf}\Lambda_n$ as the preimage of $G$ in $\ocC$ under $\mathscr{F}$ is a classical generator \cite{Generators}, and there is an isomorphism 
	\[
	H^i(\Lambda^G) \rightarrow D(\Lambda_n)(G,G[i])
	\]
	for all $i \in \mathbb{Z}$ by the construction of $\Lambda^G$ and the equivalence $\mathscr{F}$.
	Therefore, we may apply \Cref{Lem: AB tri equivalence} to see that
	\[
	- \otimes^{\mathbb{L}}_{\Lambda^G} G : D(\Lambda^G) \rightarrow D(\Lambda_n)
	\]
	is a $k$-linear equivalence of triangulated categories.
\end{proof}

Our final result is an immediate consequence of  \Cref{Thm: Derived Equiv}.

\begin{corollary}
	Let $G,G' \in \ocC$ be two limit generators.
	Then the piano algebras $\Lambda^G$ and $\Lambda^{G'}$ are derived equivalent.
\end{corollary}

We end this section by asking the following question.

\begin{question}
	Let $(\Sigma,M,P)$ be a marked surface, and let $\Gamma_1$ and $\Gamma_2$ be piano algebras coming from extended admissible dissections of $(\Sigma,M,P)$.
	When can we say that $\Gamma_1$ and $\Gamma_2$ are derived equivalent?
\end{question}

\printbibliography

\end{document}